\documentclass{elsarticle}

\usepackage{geometry}
\usepackage{graphicx}
\usepackage{amssymb}
\usepackage{amsmath}
\usepackage{amsthm}
\usepackage{empheq}
\usepackage{enumitem}
\usepackage{mdframed}
\usepackage{booktabs}
\usepackage{lipsum}
\usepackage{graphicx}
\usepackage{color}
\usepackage{psfrag}
\usepackage{pgfplots}
\usepackage{bm}
\usepackage{url}
\usepackage{array}
\usepackage{float}
\usepackage{mathrsfs}
\usepackage{tikz}
\usepackage{subcaption}
\usetikzlibrary{calc}

\graphicspath{{fig/}}

\DeclareMathAlphabet{\mathpzc}{OT1}{pzc}{m}{it}

\biboptions{sort&compress} 
\journal{ }
\usepackage[colorlinks=true]{hyperref}
\AtBeginDocument{
  \hypersetup{
    linkcolor=red,
    citecolor=green,
    urlcolor=magenta,
  }
}

\definecolor{ocre}{RGB}{243,102,25}
\definecolor{mygray}{RGB}{243,243,244}
\definecolor{deepGreen}{RGB}{26,111,0}
\definecolor{shallowGreen}{RGB}{235,255,255}
\definecolor{deepBlue}{RGB}{61,124,222}
\definecolor{shallowBlue}{RGB}{235,249,255}

\newtheoremstyle{mytheoremstyle}{3pt}{3pt}{\normalfont}{0cm}{\rmfamily\bfseries}{}{1em}{{\color{black}\thmname{#1}~\thmnumber{#2}}\thmnote{\,--\,#3}}
\newtheoremstyle{myproblemstyle}{3pt}{3pt}{\normalfont}{0cm}{\rmfamily\bfseries}{}{1em}{{\color{black}\thmname{#1}~\thmnumber{#2}}\thmnote{\,--\,#3}}
\theoremstyle{mytheoremstyle}
\newmdtheoremenv[linewidth=1pt,backgroundcolor=shallowGreen,linecolor=deepGreen,leftmargin=0pt,innerleftmargin=20pt,innerrightmargin=20pt,]{theorem}{Theorem}[section]
\theoremstyle{mytheoremstyle}
\newmdtheoremenv[linewidth=1pt,backgroundcolor=shallowBlue,linecolor=deepBlue,leftmargin=0pt,innerleftmargin=20pt,innerrightmargin=20pt,]{definition}{Definition}[section]
\theoremstyle{myproblemstyle}

\newtheorem{thm}{Theorem}[section]

\newtheorem{lem}{Lemma}[section]
\newtheorem{prop}{Proposition}[section]
\newtheorem{rmk}{Remark}[section]

\theoremstyle{definition}

\theoremstyle{remark}
\newcommand{\ben}{
\begin{eqnarray}}
  \newcommand{\een}{
\end{eqnarray}}
\newcommand{\bea}{
\begin{array}}
  \newcommand{\eea}{
\end{array}}

\numberwithin{equation}{section}
\usepgfplotslibrary{colorbrewer}
\pgfplotsset{width=8cm,compat=1.9}

\def\b0{\mathbf{0}}

\title{An Efficient Solver for the Richards Equation for
Variably Saturated Flows in Porus Media }

\begin{document}
\begin{frontmatter}
  \author{Xuelong Gu$^{a}$}
  \author{and Qi Wang$^{a,*}$}

  \address[1]{Department of Mathematics, University of South
  Carolina, Columbia, SC, 29208, USA}


  \begin{abstract}
    We present a nonlinear multigrid solver for the Richards equation
    in variably saturated porous media with strongly nonlinear
    hydraulic conductivity and water-retention relationships. The
    governing equation is discretized using a second-order
    conservative finite-difference scheme in space and an implicit
    backward difference formula in time. The core of the solver is a Nonlinear
    Gauss-Seidel (NGS) smoother based on a triangular splitting of
    the diffusion operator and diagonal stabilization.
    This splitting leads to an update scheme consisting of a sequence
    of locally decoupled scalar nonlinear problems, each of which can
    be solved efficiently and robustly using a few Newton iterations.
    Under monotonicity
    assumptions, we establish $L^\infty$-convergence of the NGS
    iteration and derive explicit conditions for the stabilization
    parameters. Numerical results for benchmark infiltration,
    drainage, and root-uptake problems demonstrate that the NGS-based
    multigrid framework is both computationally efficient and robust.

  \end{abstract}


\end{frontmatter}

\begin{figure}[b]
  \small \baselineskip=10pt
  \rule[2mm]{1.8cm}{0.2mm} \par
  $^{*}$Corresponding author.\\
  E-mail address: \url{qwang@math.sc.edu} (Q. Wang).
\end{figure}

\section{Introduction}


The Richards equation has been used to model the spatiotemporal
dynamics of root zone (e.g., top 1 m of soil) soil moisture from
precipitation and surface soil moisture inpact for several
decades \cite{Richards1931}. It describes irrigation, precipitation,
evapotranspiration, runoff, and drainage dynamics in the soil zone in
the following form
\begin{equation} \label{eq:model}
  \left \{\bea{l}
    \partial_t \theta(\psi) + \nabla \cdot \mathbf{q} = -S(\psi), \\\\
    \mathbf{q} = -K(\theta(\psi)) \nabla(\psi + z),
    \eea\right.
  \end{equation}
  where, \(\psi\) represent the pressure head (m), \(q\)
  represents the water flux (m\(^3\)/m\(^2\cdot\)s), \(S\) is the
  source/sink term associated with root water uptake (s\(^{-1}\)),
  \(\theta\) denotes the soil moisture content (m\(^3\)/m\(^3\)), \(K\)
  is the unsaturated hydraulic conductivity (m/s), \(t \in [0,T]\) is
  time, and \(z\) denotes the vertical depth (m).
  The Richards equation is a nonlinear convection-diffusion equation,
  in which the flux term arises from the Darcy's law with gravity.


  The nonlinearity in \eqref{eq:model} is entirely encoded in the
  constitutive relations $\theta(\psi)$ and $K(\psi)$ (or $K(\theta)$),
  which are specified through empirical water-retention curves (WRCs)
  and hydraulic-conduictivity functions (HCFs). Common choices include
  the Gardner model \cite{gardner1958}, the Haverkamp formulation
  \cite{haverkamp1977}, and the widely used Mualem–van
  Genuchten model \cite{mualem1976, vangenuchten1980};
  see Table~\ref{tab:models} for a few examples.
  \begin{table}[H]
    \centering
    \begin{tabular}{|c| c| c|}
      \toprule
      \textbf{Model} & \textbf{HCF \(K(\psi)\) or \(K(\theta)\)} &
      \textbf{WRC \(\theta(\psi)\)} \\
      \midrule

      Haverkamp et al. (1977) &
      $ K_s \tfrac{A}{A + |\psi|^\gamma} $ &
      $ \theta_r + \tfrac{\alpha(\theta_s - \theta_r)}{\alpha +
      |\psi|^\beta} $ \\
      \midrule
      Mualem (1976); &&\\
      Van Genuchten (1980) &
      $ K_s \sqrt{ \tfrac{\theta - \theta_r}{\theta_s - \theta_r} }
      \left[ 1 - \left( 1 - \left( \tfrac{\theta - \theta_r}{\theta_s -
      \theta_r} \right)^{1/m} \right)^m \right]^2 $ &
      \( \theta_r +
        \tfrac{\theta_s - \theta_r}{\left(1 + (\alpha
      |\psi|)^n\right)^{\frac{n-1}{n}}} \) \\
      \midrule
      Gardner (1958) &
      \( K_s e^{\alpha \psi}, \psi<0; K_s ,\psi\geq 0 \)&
      \( \theta_r + (\theta_s - \theta_r)e^{\alpha \psi}, \psi<0;
      \theta_s, \psi\geq 0\) \\
      \bottomrule
    \end{tabular}
    \caption{Commonly used hydraulic conductivity (HCF) and water
    retention curve (WRC) models.}
    \label{tab:models}
  \end{table}
  These models are highly nonlinear and often exhibit rapid variations
  of $\theta$ and $K$ over relatively small ranges of $\psi$,
  especially near saturation or under very dry conditions.  As a
  consequence, analytic solutions of the Richards equation are rarely
  available; so, efficient and  robust numerical
  approximations are essential.

  Over the past decades, a broad spectrum of discretizations and solution
  strategies have been developed for the Richards equation, including
  finite difference \cite{celia1990, Gasiorowski2020}, finite volume
  \cite{Song2026, Eymard1999, Lai2015},
  and finite element methods \cite{celia1990}, as well as mixed
  and mass‑conservative variants. In fully discrete form these
  approaches lead, at each time level, to large systems of nonlinear
  algebraic equations whose coefficients are strongly heterogeneous and
  can vary by several orders of magnitude in space and time. The
  resulting systems are challenging for standard nonlinear solvers:
  Picard or fixed‑point iterations \cite{Song2026, Gasiorowski2020,
  celia1990} may
  require very small time steps or heavy damping to converge, while
  Newton‑type methods, though typically faster when they converge, can
  become fragile in the presence of sharp wetting fronts, highly
  heterogeneous soils, or strongly nonlinear sink terms. Ensuring
  robustness across a wide range of soil types and hydrologic regimes,
  while retaining high spatial accuracy and computational efficiency,
  therefore remains an active research topic.

  In parallel to these grid-based discretization approaches, there has
  recently been growing interest in physics-informed neural networks
  (PINNs) as mesh-free, data-integrated solvers for the
  Richards equation. PINN formulations approximate the
  pressure head and related quantities by neural networks whose loss
  functions penalize both mismatch with measured data and violation of
  the governing PDE, boundary, and initial conditions in $L_2$ norms.
  For example,
  Bandai and Ghezzehei \cite{Bandai2021PINN,Bandai2022HESS} proposed
  PINN frameworks with monotonicity constraints and domain
  decomposition to estimate WRCs and HCFs, reconstruct soil-water fluxes,
  and handle layered soils with discontinuous conductivities using sparse
  soil moisture observations. Chen et al.\ \cite{Chen2023PINN} designed
  principled loss-weighting strategies to improve the robustness of
  RRE-solving PINNs, while Haruzi and Moreno
  \cite{Haruzi2023WRR} trained PINNs with geoelectrical data to model
  unsaturated flow and solute transport. Although the PINN approach
  is easy and straightforward to implement and versatile in handling
  complex geometries and boundary conditions, but it suffers from low
  order of accuracy and slow training time. It under-performs in
  simple geometries and classical boundary conditions while compared
  with the good traditional methods.

  More recently, workflow-type
  PINN models have been developed for forward and inverse modeling of
  unsaturated flow and root water uptake from hydrogeophysical data
  \cite{Sakar2025JHydrol}. These studies demonstrate the potential of
  PINNs as flexible surrogates and a straightforward formulation for
  inverse problems concerning the Richards
  equation. However, they also highlight challenges such as sensitivity to
  loss balancing and random initialization, substantial training
  costs,  slow convergence and low accuracy. In inverse problems for
  identifying conductivity coefficients, the Richards equation must
  be solved repeatedly. An efficient and robust solver is thus
  essential to ensure the accuracy of the solution process. In this
  paper we focus on developing robust, mass-conservative
  grid-based discretizations and iterative solvers for the Richards
  equation, which can also serve
  as benchmarks and building blocks for future inverse problem solver
  based on PINN developments or Kalman filter/SLE approaches \cite{Yeh-SLE}.

  In this work, we consider a conservative second‑order finite‑difference
  discretization of the Richards equation in space combined with
  backward differentiation formulas
  in time. The spatial discretization is based on a standard
  cell‑centered grid with face‑centered averages of the hydraulic
  conductivity, leading to a discrete diffusion operator that is
  symmetric, positive definite, and locally conservative. The temporal
  discretization employs a fully implicit backward Euler (BDF1) scheme.
  In our formulation, the storage term $\theta(\psi)$ and
  the sink term $S(\psi)$ are treated implicitly at the new time level,
  while the diffusion operator $-\nabla\cdot(K\nabla(\bullet))$ is
  evaluated at a lagged state. This yields at each time step a nonlinear
  system in the pressure head whose nonlinearity is confined locally
  and admits a favorable monotonicity structure under
  assumptions on $\theta$ and $S$.

  A main challenge is how to construct a nonlinear solver that is (i) robust
  for strongly nonlinear WRC/HCF pairs and realistic root‑uptake
  functions, (ii) efficient on fine grids and in higher dimensions, and
  (iii) amenable to rigorous convergence analysis. To this end, we
  propose a nonlinear Gauss–Seidel (NGS) smoother based on a triangular
  splitting of the discrete diffusion operator together with local
  diagonal shifts. More precisely, we decompose the discrete operator
  $\mathbf{M}$ associated with $-\nabla\cdot(K\nabla(\bullet))$ into its
  strictly lower, diagonal, and strictly upper triangular parts,
  \begin{equation}
    \mathbf{M} = \mathbf{L} + \mathbf{D} + \mathbf{U},
  \end{equation}
  and introduce a diagonal perturbation
  $\bm{\Lambda}_\kappa$ to enhance diagonal dominance. This leads to a
  modified splitting
  \begin{equation}
    \mathbf{M} = \mathbf{L} + \mathbf{D}_\kappa + \mathbf{U}_\kappa,
  \end{equation}
  where $U_{\kappa}$ is an upper triangular matrix, which in turn
  yields a pointwise NGS iteration where each update
  reduces to solving a scalar nonlinear equation. A stabilization
  parameter $\xi > 0$ is incorporated in the temporal discretization to
  further strengthen the diagonal dominance and improve convergence of the local
  Newton solver. Under natural monotonicity assumptions on the
  constitutive laws, we prove that the resulting NGS iteration is a
  contraction in the discrete $L^\infty$ norm, provided $\xi$ and the
  diagonal shifts $\bm{\Lambda}_\kappa$ satisfy explicit bounds derived
  from the row sums of $\mathbf{U}_\kappa$. This analysis also offers
  practical guidelines for choosing $\xi$ and $\bm{\Lambda}_\kappa$ so
  as to balance robustness and efficiency.

  To further accelerate the solution of the fully discrete Richards
  system, we embed the proposed NGS iteration into a nonlinear multigrid
  method based on the full approximation scheme (FAS) \cite{Henson2003,
  Briggs2000}.
  On each grid level, the same pointwise NGS update is used as a
  relaxation process, while FAS coarse-grid corrections treat the
  remaining low-frequency components of the error. As the result, the NGS
  scheme and FAS are combined into a robust nonlinear solver for the
  Richards equation. Numerical experiments on benchmark
  infiltration, drainage, and root-uptake cases show that the
  resulting NGS–FAS algorithm converges reliably over a wide range of
  soil parameters and hydrologic regimes.

  The remainder of the paper is organized as follows. In
  Section~\ref{sec:discretization} we present the spatial
  finite‑difference discretization and the fully discrete BDF schemes
  for the Richards system, and we summarize the structural properties of
  the resulting discrete operators. Section~\ref{sec:smoother} introduces
  the nonlinear Gauss–Seidel smoother. Then, we establish convergence
  of the smoother in the discrete $L^\infty$ norm under suitable
  monotonicity assumptions on the constitutive relations. Numerical
  experiments are reported in Section~\ref{sec:experiment}, where we
  examine the accuracy and efficiency of the proposed methods on a
  variety of test problems. Finally, Section~ summarizes the results and
  discusses possible extensions.

  \section{Fully discrete Richards system}\label{sec:discretization}

  In this section, we develop a fully discrete scheme for the Richards
  system. For definiteness, the spatial discretization is presented
  with a second-order finite‑difference method; the analysis presented
  in subsequent sections rely only on structural properties of the
  resulting discrete operators, and therefore extends verbatim to other
  spatial discretizations that share these properties.

  We impose homogeneous Dirichlet boundary conditions for simplicity.
  Let $d\in\{1,2,3\}$ and $\Omega=\prod_{\ell=1}^d
  (0,L_\ell)\subset\mathbb{R}^d$. For each direction $\ell$, choose
  $N_\ell\in\mathbb{N}$ interior nodes and set
  $h_\ell:=L_\ell/(N_\ell+1)$.
  Define the index set and interior nodes
  \begin{equation}
    \mathcal{I}=\prod_{\ell=1}^d\{1,2,\dots,N_\ell\},\quad
    x_{\bm{i}}=((i_1 + \tfrac{1}{2})h_1,\dots,(i_d +
    \tfrac{1}{2})h_d),\quad \bm{i}=(i_1,\dots,i_d)\in\mathcal{I}.
  \end{equation}
  We denote $\psi_{\bm{i}}$ as the approximation of $\psi(x_{\bm{i}})$ and
  $\bm{e}_\ell$ as the $\ell$-th canonical unit vector in
  $\mathbb{Z}^d$. For each interior face orthogonal to direction
  $\ell$, we define the face average
  \begin{equation}
    K_{\bm{i}+\frac{1}{2}\bm{e}_\ell}
    := \tfrac{1}{2} \left[K\Big(\psi_{\bm{i}}\Big) +
    K\left(\psi_{\bm{i} + \frac{1}{2} \bm{e}_\ell} \right) \right],
    \quad 1\le i_\ell\le N_\ell-1,
  \end{equation}
  The second-order conservative finite-difference semi-discretization
  of \eqref{eq:model} then reads, for all $\boldsymbol{i}\in\mathcal{I}$,
  \begin{equation} \label{eq:discrete-strong}
    \partial_t \theta(\psi_{\bm{i}})-\sum_{\ell=1}^d
    \tfrac{1}{h_\ell}\left[
      K_{\bm{i}+\frac{1}{2}\bm{e}_\ell}\,
      \tfrac{(\psi + z)_{\bm{i}+\bm{e}_\ell}-(\psi + z)_{\bm{i}}}{h_\ell}
      - K_{\bm{i}-\frac{1}{2}\bm{e}_\ell}\,
      \tfrac{(\psi + z)_{\bm{i}}-(\psi + z)_{\bm{i}-\bm{e}_\ell}}{h_\ell}
    \right] = -S \left(\psi_{\bm{i}}\right).
  \end{equation}

  Let $N_{\rm dof}:=\prod_{\ell=1}^d N_\ell$. We flatten the grid in
  lexicographic order with the first index fastest as follows:
  \begin{equation}
    s(\boldsymbol{i}) \;=\; 1+\sum_{\ell=1}^d (i_\ell-1)\prod_{k=1}^{\ell-1}N_k.
  \end{equation}
  We define the flatted vector $\vec{\bm{\psi}}$ by
  $\vec{\bm{\psi}}_{s(\boldsymbol{i})}=\psi_{\boldsymbol{i}}$ and
  introduce the diffusion matrix
  $\mathbf{M}\left(\vec{\bm{\psi}}\right)$ corresponding to the
  discrete operator $- \nabla \cdot (K \nabla (\bullet))$, defined
  componentwise for any vector $\vec{\bm{\phi}} \in
  \mathbb{R}^{N_{\rm dof}}$, as follows
  \begin{equation}
    \left(\mathbf{M}\left(\vec{\bm{\psi}}\right) \cdot
    \vec{\bm{\phi}}\right)_{s(\bm{i})} = -\sum_{\ell=1}^d
    \tfrac{1}{h_\ell}\left[
      K_{\bm{i}+\frac{1}{2}\bm{e}_\ell}\,
      \tfrac{(\phi + z)_{\bm{i}+\bm{e}_\ell}-(\phi + z)_{\bm{i}}}{h_\ell}
      - K_{\bm{i}-\frac{1}{2}\bm{e}_\ell}\,
      \tfrac{(\phi + z)_{\bm{i}}-(\phi + z)_{\bm{i}-\bm{e}_\ell}}{h_\ell}
    \right].
  \end{equation}
  With this notation, the semi-discrete system recasts into  the
  following compact form
  \begin{equation}\label{eq:linear-system}
    \partial_t \theta(\vec{\bm{\psi}}) = - S
    \left(\vec{\bm{\psi}}\right) -
    \mathbf{M}\left(\vec{\bm{\psi}}\right) \cdot \left(\vec{\bm{\psi}}
    + \vec{\bm{z}}\right).
  \end{equation}

  We now turn to the temporal discretization. Let $T > 0$. We consider
  the uniform grid $0 = t_0 < t_1 < \cdots < t_{N_t} = T$ with step
  size $\tau = T / N_t$. Denote $\vec{\bm{\psi}}^n$ the approximation
  of $\vec{\bm{\psi}}(t_n)$. A first-order backward difference (BDF1)
  scheme for \eqref{eq:linear-system} reads
  \begin{equation} \label{eq:fully-bdf1}
    \tfrac{ \theta\left(\vec{\bm{\psi}}^{n+1}\right) -
    \theta\left(\vec{\bm{\psi}}^n\right) }{\tau} = -S
    \left(\vec{\bm{\psi}}^{n+1}\right) - \xi
    \left(\vec{\bm{\psi}}^{n+1} - \vec{\bm{\psi}}^{n}\right) -
    \mathbf{M}\left(\vec{\bm{\psi}}^{n}\right) \cdot
    \left(\vec{\bm{\psi}}^{n+1} + \vec{\bm{z}}\right).
  \end{equation}

  In \eqref{eq:fully-bdf1}, $\xi > 0$ is a
  stabilization parameter introduced to enhance the convergence of the
  nonlinear solver used below. While larger values of $\xi$ increase
  robustness, they also introduce extra damping temporally; so in
  practice, $\xi$ should be chosen as small as possible without
  compromising convergence. We will show below that for $\xi > 0$, the
  proposed iterative scheme converges provided the constitutive laws satisfy
  the following monotonicity assumptions:

  \begin{enumerate}[label=\textup{(A\arabic*)}, ref=\textup{(A\arabic*)}]
    \item \label{assump:A1} $\theta'(\psi)\ge c_0>0$;
    \item \label{assump:A2} $S'(\psi)\ge 0$ for
      $\psi\in[\psi_{\mathrm{start}},\psi_{\mathrm{opt}}]$,
  \end{enumerate}
  i.e., the specific moisture capacity $\theta^\prime(\psi)$ is
  positive and the root-water-uptake term increases with $\psi$ on the
  rising branch from the onset to the optimal take. \ref{assump:A1}
  yields strict parabolicity and strong monotonicity of the storage
  term, which underpins stability and uniqueness; \ref{assump:A2}
  provides a monotone sink, which contributes to contractivity of the
  nonlinear map. Outside interval $[\psi_{start}, \psi_{opt}]$ (e.g.,
  in the over-wet/anaerobic regime), $S^\prime(\psi)$ may change sign,
  and then a positive $\xi$ is used to compensate the nonlinearity
  induced by $\theta(\bullet)$ and $S(\bullet)$.

  \section{Efficient smoother and multigrid solver}\label{sec:smoother}

  In this section, we discibe the nonlinear multigrid solver for the
  fully discrete system \eqref{eq:fully-bdf1}. We begin with
  some notations. We denote the entires of  matrix
  $\mathbf{M}\left(\vec{\bm{\psi}}^n\right)$ by
  \begin{equation}\label{eq:m_general}
    \mathbf{M} =
    \begin{pmatrix}
      m_{1, 1}            & -m_{1, 2}           & \cdots & -m_{1,
      N_{\rm dof}}                                                   \\
      -m_{2, 1}           & m_{2, 2}            & \cdots & -m_{2,
      N_{\rm dof}}                                                   \\
      \vdots              & \vdots              &        & \vdots
      \\
      -m_{N_{\rm dof}, 1} & -m_{N_{\rm dof}, 2} & \cdots & m_{N_{\rm
      dof}, N_{\rm dof}}
    \end{pmatrix}.
  \end{equation}
  We have the following observation of the $\mathbf{M}$ obtained by the
  finite difference method in \eqref{eq:fully-bdf1}.
  \begin{prop}
    For the matrix $\mathbf{M}$ obtained by the finite difference
    method in \eqref{eq:fully-bdf1}, we have
    \begin{equation}
      m_{i, i} > 0, \quad m_{i, j} \geq 0,  \quad \sum\limits_{j =
      1}^{N_{\rm dof}} m_{i,j} = m_{i, i}.
    \end{equation}
  \end{prop}

  We define the discrete $L^\infty$ norm for a vector and the induced
  matrix norm
  \begin{equation}\label{eq:def-l-infty-norm}
    \left\|\vec{\bm{ \psi }} \right\|_\infty = \max\limits_{1 \leq i
    \leq N_{\rm dof}} |\psi_i|, \quad \left\|\mathbf{M}\right\|_\infty
    = \max\limits_{\vec{\bm{\psi}} \in \mathbb{R}^{N_{\rm dof}}}
    \tfrac{\left\|\mathbf{M}\vec{\bm{\psi}}\right\|_\infty}{\left\|\vec{\bm{\psi}}\right\|_\infty}
    = \max\limits_{1 \leq i \leq N_{\rm dof}} \left\{ \sum\limits_{j =
    1}^{N_{\rm dof}} |m_{ij}| \right\}.
  \end{equation}
  We also denote the extremal row sums by
  \begin{equation}
    \begin{aligned}
      & \overline{\lambda}_\infty(\mathbf{M}) = \max\limits_{1 \leq i
      \leq N_{\rm dof}} \left\{ \sum\limits_{j = 1}^{N_{\rm dof}}
      |m_{ij}| \right\}, \quad
      \underline{\lambda}_\infty(\mathbf{M}) = \min\limits_{1 \leq i
      \leq N_{\rm dof}} \left\{ \sum\limits_{j = 1}^{N_{\rm dof}}
      |m_{ij}| \right\}.                                               \\
    \end{aligned}
  \end{equation}

  We introduce the following useful Lemma \ref{lem:decompos}
  \begin{lem} \label{lem:decompos}
    We decompose the matrix $\mathbf{M}$ into its strictly lower,
    diagonal, and strictly upper triangular parts as
    \begin{equation}\label{eq:decompos-01}
      \mathbf{M} = \mathbf{L} + \mathbf{D} + \mathbf{U},
    \end{equation}
    To enhance diagonal dominance, we introduce a diagonal shift
    \begin{equation}
      \bm{\Lambda}_\kappa = {\rm diag}\{\kappa_1, \cdots,
      \kappa_{N_{\rm dof}}\},
    \end{equation}
    and define the modified splitting
    \begin{equation}
      \mathbf{M} = \mathbf{L} + (\mathbf{D} + \bm{\Lambda}_\kappa) +
      (\mathbf{U} - \bm{\Lambda}_\kappa) =: \mathbf{L} +
      \mathbf{D}_\kappa + \mathbf{U}_\kappa.
    \end{equation}
    We further introduce another diagonal matrix $\bm{\Lambda}_\mu = \{
    \mu_1, \cdots, \mu_{N_{\rm dof}} \}$. For
    $\kappa_i > 0$, $\mu_i > 0$, $i = 1, \cdots, N_{\rm dof}$ and constants
    $\alpha, \beta > 0$, the following bounds hold:
    \begin{itemize}
      \item $\| \alpha \mathbf{I} - \beta \mathbf{U}_\kappa \|_\infty =
        \alpha + \beta \overline{\lambda}_\infty (\mathbf{U}_\kappa)$.
      \item $\left\| \left(\alpha \bm{\Lambda}_\mu + \beta (\mathbf{L}
        + \mathbf{D}_\kappa)\right)^{-1} \right\|_\infty \leq
        \tfrac{1}{\alpha \underline{\lambda}_\infty(\bm{\Lambda}_\mu) +
        \beta \underline{\lambda}_\infty (\mathbf{U}_\kappa)}$.
    \end{itemize}
  \end{lem}

  \begin{proof}
    The first estimate follows immediately from the definition of
    discrete $L^\infty$ norm \eqref{eq:def-l-infty-norm}. We prove the
    second bound. Consider the linear system
    \begin{equation}
      \Big(\alpha \bm{\Lambda}_\mu + \beta (\mathbf{L} +
      \mathbf{D}_\kappa)\Big) \vec{\bm{\psi}} = \vec{\bm{f}}.
    \end{equation}

    The coefficient matrix is lower triangular with strictly positive
    diagonal entries $\alpha\mu_i+\beta(d_{ii}+\kappa_i)$, hence it is
    non-singular.  We only need to show
    \begin{equation}
      \left\|\vec{\bm{\psi}}\right\|_\infty \leq \tfrac{1}{\alpha
        \underline{\lambda}_\infty(\bm{\Lambda}_\mu) + \beta
      \underline{\lambda}_\infty (\mathbf{U}_\kappa)}
      \left\|\vec{\bm{f}}\right\|_\infty,
    \end{equation}
    which is equivalent to
    \begin{equation}
      \left\| \left(\alpha \bm{\Lambda}_\mu + \beta (\mathbf{L} +
      \mathbf{D}_\kappa)\right) \vec{\bm{\psi}} \right\|_\infty \geq
      \left(a \underline{\lambda}_\infty (\bm{\Lambda}_\mu) + \beta
      \underline{\lambda}_\infty \left(\mathbf{U}_\kappa\right) \right)
      \left\| \vec{\bm{\psi}} \right\|_\infty.
    \end{equation}
    Suppose $\left|\psi_p\right| =
    \left\|\vec{\bm{\psi}}\right\|_\infty$, if $\psi_p \geq 0$, we have
    \begin{equation}
      \begin{aligned}
        & \left\| \left(\alpha \bm{\Lambda}_\mu + \beta (\mathbf{L} +
        \mathbf{D}_\kappa) \right) \vec{\bm{\psi}} \right\|_\infty
        \geq (\alpha \mu_p + \beta(\kappa_p + m_{pp}) ) \psi_p - \beta
        \sum\limits_{j = 1}^{p-1} m_{pj} \psi_j
        \\
        & \quad \geq (\alpha \mu_p + \beta(\kappa_p + m_{pp}) ) \psi_p
        - \beta \sum\limits_{j = 1}^{p-1} m_{pj} \psi_p
        \geq \left(\alpha \mu_p + \beta \kappa_p + \beta
        \sum\limits_{j=p+1}^{N_{\rm dof}} m_{pj} \right) \psi_p         \\
        & \quad \geq (\alpha \underline{\lambda}_\infty
          (\bm{\Lambda}_\mu) + \beta
        \underline{\lambda}_\infty(\mathbf{U}_\kappa) )
        \left\|\vec{\bm{\psi}}\right\|_\infty.
      \end{aligned}
    \end{equation}
    A similar procedure yields the result for the case $\psi_p \leq 0$.
    The proof is thus completed.
  \end{proof}

  \subsection{Nonlinear Gauss-Seidel smoother and convergence}
  Given $\vec{\bm{\psi}}^n$, the nonlinear Gauss-Seidel (NGS) smoother
  for solving \eqref{eq:fully-bdf1} reads:
  \begin{itemize}
    \item Set $\vec{\bm{\psi}}^{(0)} = \vec{\bm{\psi}}^n$.
    \item For $s = 1, 2, \ldots$, compute $\vec{\bm{\psi}}^{(s+1)}$
      iteratively from
      \begin{equation}\label{eq:iterative}
        \begin{aligned}
          \tfrac{ \theta\left(\vec{\bm{\psi}}^{(s+1)}\right) -
          \theta\left(\vec{\bm{\psi}}^n\right) }{\tau} & = -S
          \left(\vec{\bm{\psi}}^{(s+1)}\right) -
          \xi\left(\vec{\bm{\psi}}^{(s+1)} - \vec{\bm{\psi}^n}\right) -
          \mathbf{U}_\kappa \left(\vec{\bm{\psi}}^n\right)  \cdot
          \left(\vec{\bm{\psi}}^{(s)} + \vec{\bm{z}}\right)
          \\
          & - \left[\mathbf{L}\left(\vec{\bm{\psi}}^n\right)+
          \mathbf{D}_{\kappa}\left(\vec{\bm{\psi}}^n\right) \right]\cdot
          \left(\vec{\bm{\psi}}^{(s+1)} + \vec{\bm{z}}\right),
        \end{aligned}
      \end{equation}
      ant stop when the nonlinear residual
      \begin{equation}
        \vec{\bm{r}}^{(s+1)} = - S\left(\vec{\bm{\psi}}^{(s+1)}\right)
        - \xi \left(\vec{\bm{\psi}}^{(s+1)} - \vec{\bm{\psi}}^n\right)
        - \mathbf{M}\left(\vec{\bm{\psi}}^n\right) \cdot
        \left(\vec{\bm{\psi}}^{(s+1)} + \vec{\bm{z}}\right) -
        \tfrac{\theta\left(\vec{\bm{\psi}}^{(s+1)}\right) -
        \vec{\bm{\psi}}^n}{\tau}.
      \end{equation}
      satisfies the stop criteria
      $\left\|\vec{\bm{r}}^{(s+1)}\right\|_\infty < \textit{tol}$.
  \end{itemize}
  Although \eqref{eq:iterative} is fully implicit, the triangular
  splitting yields a pointwise implementation:
  in a forward sweep, the $i$-th update involves a single scalar
  unknown $\psi_i^{(s+1)}$, leading to the following scalar nonlinear equation:
  \begin{equation}\label{eq:point-update}
    \tfrac{\theta\left(\psi_i^{(s+1)}\right)}{\tau} + S(\psi_i^{(s+1)})
    + (\xi + \kappa_i + m_{ii}) \psi_i^{(s+1)} = f_i,
  \end{equation}
  with the known right-hand side
  \begin{equation}
    f_i  = \tfrac{\theta\left(\psi_i^n\right)}{\tau} + \xi \psi_i^n +
    \sum\limits_{j = 1}^{i-1} m_{ij} \psi_j^{(s+1)} + \sum\limits_{j =
    i+1}^{N_{\rm dof}} m_{ij} \psi_j^{(s)} + \kappa_i \psi_i^{(s)}.
  \end{equation}
  Under Assumptions \ref{assump:A1}-\ref{assump:A2}, the left-hand side
  of \eqref{eq:point-update} is strictly increasing in $\psi_i^{(s+1)}$,
  so each point update is uniquely defined and can be obtained by a few
  (damped) Newton steps with optional backtracking.
  Intuitively, the shift $\kappa_i$ weakens the coupling carried by
  $\mathbf{U}^n$ and increases the effective diagonal, which improves
  the smoothing factor for high-frequency error modes.

  \begin{figure}[H]
    \centering
    \begin{tikzpicture}[scale=0.55,inner sep=0pt,minimum size=2mm,thick]

      \draw[dotted] (0,0) -- (12,0);
      \node (A) at (0,0) [circle,draw=green!50,fill=green!20] {};
      \node (B) at (4,0) [circle,draw=green!50,fill=green!20] {};
      \node (C) at (6,0) [circle,draw=orange!50,fill=orange!20] {};
      \node (D) at (8,0) [circle,draw=red!50,fill=red!20] {};
      \node (E) at (12,0) [circle,draw=red!50,fill=red!20] {};

      \path[->,dotted] (A) edge (B);
      \path[->,solid] (B) edge (C);
      \path[->,solid] (C) edge (D);
      \path[->,dotted] (D) edge (E);

      \node[below=4pt, font=\tiny] at (A) {$0$};
      \node[below=4pt, font=\tiny] at (B) {$i-1$};
      \node[below=4pt, font=\tiny] at (C) {$i$};
      \node[below=4pt, font=\tiny] at (D) {$i+1$};
      \node[below=4pt, font=\tiny] at (E) {$N$};

    \end{tikzpicture}
    \caption{Pointwise sweep order of the NGS smoother (1D).
      Green nodes have been updated using the newest values (forward
      Gauss-Seidel);
      the orange node is currently solved from~\eqref{eq:point-update};
    red nodes are pending updates.}
  \end{figure}
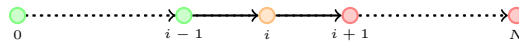

  \begin{thm}
    Assume $\left\| \vec{\bm{\psi}}^n \right\|_\infty < \infty$ and that
    $\theta(\bullet)$ and $S(\bullet)$ satisfy Assumptions
    \ref{assump:A1}-\ref{assump:A2}. Then, the nonlinear smoother
    \eqref{eq:iterative} defines a contraction in the discrete
    $L^\infty$ norm provided
    \begin{equation}\label{eq:iterative-conv}
      \xi \geq \overline{\lambda}_\infty (\mathbf{U}_\kappa) -
      \underline{\lambda}_\infty (\mathbf{U}_\kappa).
    \end{equation}
    Otherwise, there exists a sufficient large constant $\xi_\star$ such
    that for any $\xi > \xi_\star$, iteration \eqref{eq:iterative}
    is a contraction.
  \end{thm}
  \begin{rmk}
    We remark that we can choose the stabilizations parameters, $\xi$ and
    $\bm{\Lambda}_\kappa$,flexibly to balance the convergence speed and
    stability of the proposed iterative schemes. Especially when
    Assumptions.~\ref{assump:A1}-\ref{assump:A2} holds. In this case,
    the first choice is
    \begin{equation}
      \bm{\Lambda}_\kappa = {\rm diag}
      \left\{\overline{\lambda}_\infty(\mathbf{U}) -
        \sum\limits_{j=2}^{N_{\rm dof}}  m_{1,j},
        \overline{\lambda}_{\infty} (\mathbf{U}) - \sum\limits_{j =
        3}^{N_{\rm dof}} m_{2,j}, \cdots, \overline{\lambda}_\infty
        (\mathbf{U}) - m_{N_{\rm dof}-1, N_{\rm dof}},
      \overline{\lambda}_\infty (\mathbf{U}) \right\}.
    \end{equation}
    Then, $\overline{\lambda}_\infty(\mathbf{U}) =
    \underline{\lambda}_\infty(\mathbf{U})$, and condition
    \eqref{eq:iterative-conv} reduces to $\xi \geq 0$, therefore
    iteration \eqref{eq:iterative} is a contraction for any $\xi \geq
    0$ under Assumptions \ref{assump:A1}-\ref{assump:A2}.

    Another choose is to set $\xi = 0$; then, one can let  the time step
    to be sufficiently small to guarantee the convergence of the iteration.
  \end{rmk}
  \begin{proof}
    We recall the iteration in \eqref{eq:iterative}
    \begin{equation}
      \begin{aligned}
        \tfrac{ \theta\left(\vec{\bm{\psi}}^{(s+1)}\right) -
        \theta\left(\vec{\bm{\psi}}^n\right) }{\tau} & = -S
        \left(\vec{\bm{\psi}}^{(s+1)}\right) -
        \xi\left(\vec{\bm{\psi}}^{(s+1)} - \vec{\bm{\psi}}^n\right) -
        \mathbf{U}_\kappa\left(\vec{\bm{\psi}}^n\right)  \cdot
        \left(\vec{\bm{\psi}}^{(s)} + \vec{\bm{z}}\right)
        \\
        & - \left[\mathbf{L}\left(\vec{\bm{\psi}}^n\right)+
        \mathbf{D}_\kappa \left(\vec{\bm{\psi}}^n\right) \right]\cdot
        \left(\vec{\bm{\psi}}^{(s+1)} + \vec{\bm{z}}\right).
      \end{aligned}
    \end{equation}
    The solution of the nonlinear system satisfies
    \begin{equation}
      \tfrac{ \theta\left(\vec{\bm{\psi}}^{\star}\right) -
      \theta\left(\vec{\bm{\psi}}^n\right) }{\tau} = -S
      \left(\vec{\bm{\psi}}^{\star}\right) - \xi \left(
      \vec{\bm{\psi}}^\star - \vec{\bm{\psi}}^n \right) -
      \mathbf{M}\left(\vec{\bm{\psi}}^n\right) \cdot
      \left(\vec{\bm{\psi}}^{\star} + \vec{\bm{z}}\right).
    \end{equation}
    Subtracting the above two equations and involking $\mathbf{M} =
    \mathbf{L} + \mathbf{D}_\kappa\left(\vec{\bm{\psi}}^n\right) +
    \mathbf{U}_\kappa\left(\vec{\bm{\psi}^n}\right)$ yields the error equation
    \begin{equation}
      \begin{aligned}
        &
        \tfrac{1}{\tau}\left[\theta\left(\vec{\bm{\psi}}^{(s+1)}\right)
        - \theta\left(\vec{\bm{\psi}}^\star\right)\right] +
        S\left(\vec{\bm{\psi}}^{s+1}\right) -
        S\left(\vec{\bm{\psi}}^\star\right) + \xi \vec{\bm{e}}^{(s+1)} \\
        & =  - \left[ \mathbf{L}\left(\vec{\bm{\psi}}^n\right) +
        \mathbf{D}_\kappa \left( \vec{\bm{\psi}}^n \right) \right]
        \cdot \vec{\bm{e}}^{(s+1)} -
        \mathbf{U}_\kappa\left(\vec{\bm{\psi}}^n\right) \cdot
        \vec{\bm{e}}^{(s)},                                             \\
      \end{aligned}
    \end{equation}
    where $\vec{\bm{e}}^{(s)} = \vec{\bm{\psi}}^\star -
    \vec{\bm{\psi}}^{(s)}$. Utilizing the mean value theorem and making
    some arrangements, we have
    \begin{equation}
      \left(\bm{\Lambda}(\tau, \theta, S) + \xi \mathbf{I} +
        \mathbf{L}\left(\vec{\bm{\psi}}^n\right)+
      \mathbf{D}_\kappa\left(\vec{\bm{\psi}}^n\right)\right)
      \vec{\bm{e}}^{(s+1)} =  -
      \mathbf{U}_\kappa\left(\vec{\bm{\psi}}^n\right) \vec{\bm{e}}^{(s)}.
    \end{equation}
    Here,
    \begin{equation}
      \bm{\Lambda}(\tau , \theta, S) = {\rm diag} \{ \lambda_1(\tau,
      \theta, S), \cdots, \lambda_{N_{\rm dof}}(\tau, \theta, S) \},
    \end{equation}
    with
    \begin{equation}
      \lambda_i(\tau, \theta, S) =
      \tfrac{1}{\tau}\theta^\prime\left(\tilde{\psi}_i\right) +
      S^\prime\left(\bar{\psi}_i\right).
    \end{equation}
    Therefore, $\underline{\lambda}_\infty(\bm{\Lambda}(\tau, \theta,
    S)) \geq \tfrac{1}{\tau}\theta_0 + S_0$. According to
    Lemma~\ref{lem:decompos}, one has
    \begin{equation}
      \left\| \vec{\bm{e}}^{(s+1)} \right\|_\infty \leq
      \tfrac{\overline{\lambda}_\infty(\mathbf{U}_\kappa)}{\tfrac{1}{\tau}
        \theta_0 + S_0 + \xi +
      \underline{\lambda}_\infty\left(\mathbf{U}_\kappa\right) }
      \left\| \vec{\bm{e}}^{(s)} \right\|_\infty.
    \end{equation}
    The proof is thus completed.
  \end{proof}

  \subsection{Nonlinear multigrid cycle}

  \begin{figure}[H]
    \begin{tikzpicture}[scale=1]

      \begin{scope}[xshift=0cm]

        \draw[very thick] (0,0) rectangle (4,4);

        \foreach \i in {1,2,3}{
          \draw[thick] (\i,0) -- (\i,4);
          \draw[thick] (0,\i) -- (4,\i);
        }

        \foreach \i in {0.5,1.5,2.5,3.5}{
          \foreach \j in {0.5,1.5,2.5,3.5}{
            \node[blue] at (\i,\j) {$\triangle$};
          }
        }

        \node at (2,-0.7) {44 fine};

      \end{scope}

      \begin{scope}[xshift=6cm]

        \draw[very thick] (0,0) rectangle (4,4);

        \foreach \i in {1,2,3}{
          \draw[dashed] (\i,0) -- (\i,4);
          \draw[dashed] (0,\i) -- (4,\i);
        }

        \draw[thick] (2,0) -- (2,4);
        \draw[thick] (0,2) -- (4,2);

        \foreach \i in {1,3}{
          \foreach \j in {1,3}{
            \node[red,circle,draw,fill=red!20,minimum size=8pt] at (\i,\j) {};
          }
        }

        \node at (2,-0.7) {2×2 coarse};

      \end{scope}

      \begin{scope}[xshift=12cm]

        \draw[very thick] (0,0) rectangle (4,4);

        \foreach \i in {1,2,3}{
          \draw[dashed] (\i,0) -- (\i,4);
          \draw[dashed] (0,\i) -- (4,\i);
        }

        \foreach \i in {0.5,1.5,2.5,3.5}{
          \foreach \j in {0.5,1.5,2.5,3.5}{
            \node[blue] at (\i,\j) {$\triangle$};
          }
        }

        \draw[dashed,thick] (2,0) -- (2,4);
        \draw[dashed,thick] (0,2) -- (4,2);

        \foreach \i in {1,3}{
          \foreach \j in {1,3}{
            \node[red,circle,draw,fill=red!20,minimum size=8pt] at (\i,\j) {};
          }
        }

        \node at (2,-0.7) {fine coarse};

      \end{scope}
    \end{tikzpicture}
    \caption{Two-grid hierarchy for the NGS--FAS scheme.
      Left: $4\times4$ fine grid $\Omega_h$ with cell-centered unknowns
      (blue triangles). Middle: $2\times2$ coarse grid $\Omega_H$ (red
      circles) overlaid on the underlying fine grid. Right: combined view
      of fine and coarse grid points, used to define restriction $I_H^h$
    and prolongation $I_h^H$ in the two-grid FAS algorithm.}
    \label{fig:two-grid}
  \end{figure}
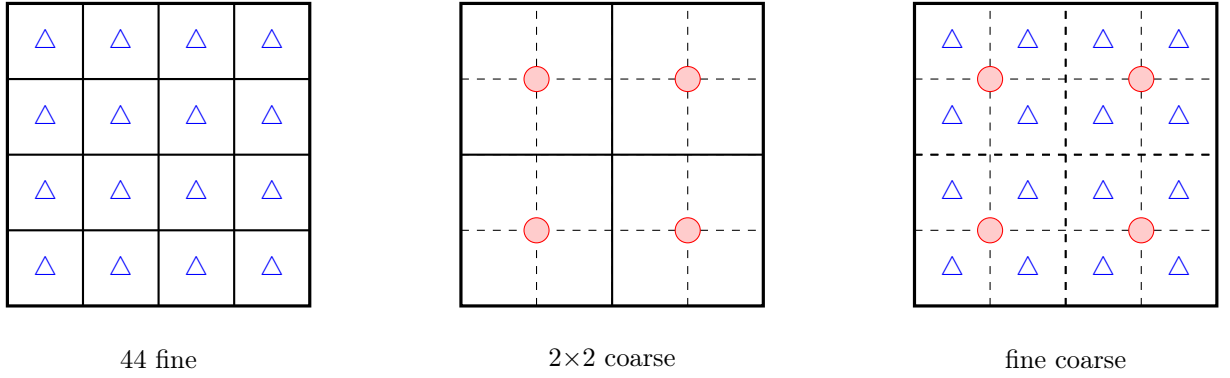

  We now embed the NGS iteration in
  \eqref{eq:iterative} into a two-grid full-approximation scheme (FAS)
  for the fully discrete Richards system defined in \eqref{eq:fully-bdf1}.
  Figure~\ref{fig:two-grid} illustrates the two-grid hierarchy used in
  the NGS-FAS algorithm. On the finest level, we consider a $4\times4$
  partition of the domain (left panel), where the blue triangles
  represent the cell-centered unknowns $\psi_{h}$ on the fine grid
  $\Omega_h$. The middle panel shows the corresponding $2\times2$ coarse
  grid $\Omega_H$ (red circles), superimposed on the underlying fine
  grid. In our implementation, coarse-grid unknowns are located at every
  other fine-grid cell center, which leads to a straightforward
  definition of the restriction operator $I_H^h$ and the prolongation
  operator $I_h^H$. The right panel combines the two sets of grid points
  and illustrates how coarse-grid corrections are injected back to the
  fine level: the NGS smoother acts on all fine-grid points (triangles),
  while the coarse-grid problem is solved for the subset of coarse
  points (circles) within the standard FAS framework.

  For a fixed time level $t_{n+1}$ we omit the
  time index and write $\vec{\bm{\psi}}$ for
  $\vec{\bm{\psi}}^{\,n+1}$ and regard $\vec{\bm{\psi}}^{\,n}$ as given.
  On a given grid we introduce the nonlinear operator
  $\mathcal{N}_h:\mathbb{R}^{N_{\rm dof}}\to\mathbb{R}^{N_{\rm dof}}$ and
  the corresponding right-hand side
  $\vec{\bm{g}}_h\in\mathbb{R}^{N_{\rm dof}}$ by
  \begin{equation}\label{eq:Nh-def}
    \mathcal{N}_h\left(\vec{\bm{\psi}}\right)
    := \tfrac{1}{\tau}\,\theta\left(\vec{\bm{\psi}}\right) +
    S\left(\vec{\bm{\psi}}\right) + \xi \vec{\bm{\psi}}
    + \mathbf{M}\left(\vec{\bm{\psi}}^{\,n}\right) \vec{\bm{\psi}},
    \quad
    \vec{\bm{g}}_h
    := \tfrac{1}{\tau}\,\theta(\vec{\bm{\psi}}^{\,n})
    + \xi \vec{\bm{\psi}}^{\,n} -
    \mathbf{M}\left(\vec{\bm{\psi}}^n\right) \vec{\mathbf{z}}.
  \end{equation}
  Then the fully discrete scheme \eqref{eq:fully-bdf1} can be written in
  the compact form
  \begin{equation}\label{eq:Nh-system}
    \mathcal{N}_h\left(\vec{\bm{\psi}}\right) = \vec{\bm{g}}_h.
  \end{equation}
  The NGS smoother \eqref{eq:iterative} is viewed as a relaxation
  operator acting on~\eqref{eq:Nh-system}: one NGS sweep corresponds to
  a forward pass of pointwise updates of \eqref{eq:point-update} over all
  degrees of freedom.

  For the two-grid construction we consider a fine grid $\Omega_h$ and a
  coarse grid $\Omega_H$ obtained by uniform coarsening (e.g., doubling
  the mesh size in each coordinate direction). The discrete operators
  $\mathcal{N}_H$ and right-hand sides $\vec{\bm{g}}_H$ are defined on
  $\Omega_H$ in complete analogy with \eqref{eq:Nh-def}. We denote by
  $I_H^h$ the restriction operator from $\Omega_h$ to $\Omega_H$ and by
  $I_h^H$ the prolongation operator from $\Omega_H$ to $\Omega_h$ (in
    the numerical experiments we use standard full-weighting restriction
  and linear interpolation, but other choices are possible).

  Given a current fine-grid approximation $\vec{\bm{\psi}}_h^{(m)}$ to
  the solution of \eqref{eq:Nh-system}, one two-grid FAS iteration
  produces an updated approximation
  $\vec{\bm{\psi}}_h^{(m+1)} = \mathrm{TG}
  \left(\vec{\bm{\psi}}_h^{(m)},\vec{\bm{g}}_h\right)$ as follows.

  \paragraph{Two-grid FAS algorithm}
  \begin{enumerate}
    \item \textbf{Pre-smoothing on the fine grid.}
      Starting from $\vec{\bm{\psi}}_h^{(m)}$, apply $\nu_1$ NGS
      sweeps to \eqref{eq:Nh-system} on $\Omega_h$:
      \begin{equation}
        \bar{\vec{\bm{\psi}}}_h
        = S_h^{\nu_1}\left(\vec{\bm{\psi}}_h^{(m)},\vec{\bm{g}}_h\right),
      \end{equation}
      where $S_h$ denotes one NGS sweep given by
      \eqref{eq:iterative}-\eqref{eq:point-update}.

    \item \textbf{Fine-grid defect.}
      Compute the nonlinear defect (FAS residual)
      \begin{equation}
        \vec{\bm{d}}_h
        = \vec{\bm{g}}_h - \mathcal{N}_h\left(\bar{\vec{\bm{\psi}}}_h\right).
      \end{equation}

    \item \textbf{Restriction to the coarse grid.}
      Restrict the smoothed approximation and the defect:
      \begin{equation}
        \bar{\vec{\bm{\psi}}}_H = I_H^h \bar{\vec{\bm{\psi}}}_h,
        \qquad
        \vec{\bm{d}}_H = I_H^h \vec{\bm{d}}_h.
      \end{equation}

    \item \textbf{Construct equation on the Coarse grid.}
      Define the FAS coarse-grid right-hand side by
      \begin{equation}
        \vec{\bm{g}}_H^{\rm FAS}
        = \mathcal{N}_H\left(\bar{\vec{\bm{\psi}}}_H\right) + \vec{\bm{d}}_H.
      \end{equation}
      The coarse-grid problem to be solved is then
      \begin{equation}\label{eq:coarse-fas}
        \mathcal{N}_H\left(\vec{\bm{\psi}}_H\right) = \vec{\bm{g}}_H^{\rm FAS}
        \quad \text{on } \Omega_H.
      \end{equation}

    \item \textbf{Approximate coarse-grid solution.}
      Starting from the initial guess $\bar{\vec{\bm{\psi}}}_H$,
      perform $\nu_H$ NGS sweeps for
      \eqref{eq:coarse-fas} on $\Omega_H$ to obtain an approximate
      coarse solution $\widehat{\vec{\bm{\psi}}}_H$.

    \item \textbf{Coarse-grid correction.}
      Construct the the coarse-grid error
      \begin{equation}
        \vec{\bm{e}}_H
        = \widehat{\vec{\bm{\psi}}}_H - \bar{\vec{\bm{\psi}}}_H,
      \end{equation}
      prolongate it to the fine grid
      \begin{equation}
        \vec{\bm{e}}_h = I_h^H \vec{\bm{e}}_H,
      \end{equation}
      and correct the fine-grid approximation:
      \begin{equation}
        \vec{\bm{\psi}}_h^{(m,\mathrm{corr})}
        = \bar{\vec{\bm{\psi}}}_h + \vec{\bm{e}}_h.
      \end{equation}

    \item \textbf{Post-smoothing on the fine grid.}
      Apply $\nu_2$ additional NGS sweeps on $\Omega_h$:
      \begin{equation}
        \vec{\bm{\psi}}_h^{(m+1)}
        = S_h^{\nu_2}\left(
          \vec{\bm{\psi}}_h^{(m,\mathrm{corr})},
          \vec{\bm{g}}_h
        \right).
      \end{equation}
  \end{enumerate}

  One two-grid FAS step is thus completely specified by the NGS smoother
  \eqref{eq:iterative}, the choice of restriction/prolongation
  operators, and the integers $(\nu_1,\nu_2,\nu_H)$. In the numerical
  tests below we monitor the nonlinear residual
  $\mathcal{N}_h(\vec{\bm{\psi}}_h^{(m)}) - \vec{\bm{g}}_h$ in the
  discrete $L^\infty$ norm \eqref{eq:def-l-infty-norm} and terminate the
  two-grid iteration once it falls below a prescribed tolerance.

  We emphasize that the above description focuses on a two-grid FAS
  scheme for clarity. Once the two-grid components (restriction,
  prolongation, and the NGS smoother) are in place, standard multilevel
  V-cycles and full multigrid (FMG) schemes follow by a recursive
  application of the same procedure on successively coarser grids
  \cite{Briggs2000}. In practice, a V-cycle is obtained by replacing
  the coarse-grid solve in the two-grid algorithm with a recursive call
  of the same FAS routine on the next coarser level, while FMG is built
  by starting from the coarsest grid, successively interpolating the
  solution to finer grids, and performing one or several V-cycles on
  each level. Since these extensions are straightforward once the
  two-grid formulation is available, we restrict the presentation to
  the two-grid NGS-FAS scheme and do not repeat the standard multilevel
  pseudocode.

  \section{Numerical results}\label{sec:experiment}

  \noindent \indent We benchmark the algorithm using a 1-d and 2-d
  classical example, respectively, to demonstrate its efficiency,
  robustness, and stability next.

  \subsection{1-d problem}
  In the first example, we consider the classical one-dimensional
  benchmark of Celia et al. on a $40
  \mathrm{cm}$-deep solid column. The HCF and the WRC curve follow the
  Haverkamp model; the corresponding parameters are listed in
  Table.~\ref{tab:case1d_haverkamp_param}. The
  initial condition is $\psi(z, 0) = -61.5 {\rm cm}$, Dirichlet
  boundary conditions are imposed at the top and bottom as follows
  \begin{equation}
    \psi(40{\rm cm}, t) = -20.7{\rm cm}, \ \psi(0, t) = -61.5{\rm cm}.
  \end{equation}
  The final time is $t = 360$ and the sink term is neglected.

  This example is used to assess the efficiency of the proposed
  nonlinear Gauss-Seidel smoother and the convergence of the
  multigrid method. Since WRC $\theta$ satisfies
  assumptions.~\ref{assump:A1}-\ref{assump:A2} (see
  Figure.~\ref{fig:case1d_k_theta_curve}), we set the stabilization
  parameter to $\xi = 0$, and employ a diagonal shift $\Lambda_\kappa = {\rm
  diag}\{0.1, \cdots, 0.1 \}$ in the implementation.

  \begin{table}[H]
    \centering
    \caption{Soil-specific parameters and values used in the 1-D case
    study of Celia et al. (1990) based on Haverkamp et al. (1977).}
    \begin{tabular}{lcc}
      \toprule
      \textbf{Soil-specific parameter} & \textbf{Value} & \textbf{Unit} \\
      \midrule
      Saturated hydraulic conductivity, \(K_s\) & 0.00944 & cm/s \\
      Saturated soil moisture content, \(\theta_s\) & 0.287 & -- \\
      Residual soil moisture content, \(\theta_r\) & 0.075 & -- \\
      \(\alpha\) in Haverkamp’s model & \(1.611 \times 10^6\) & cm \\
      \(A\) in Haverkamp’s model & \(1.175 \times 10^6\) & cm \\
      \(\beta\) in Haverkamp’s model & 3.96 & -- \\
      \(\gamma\) in Haverkamp’s model & 4.74 & -- \\
      Total time, \(T\) & 360 & s \\
      \bottomrule
    \end{tabular}

    \label{tab:case1d_haverkamp_param}
  \end{table}

  \begin{figure}[H]
    \begin{center}
      \includegraphics[width=0.48\textwidth]{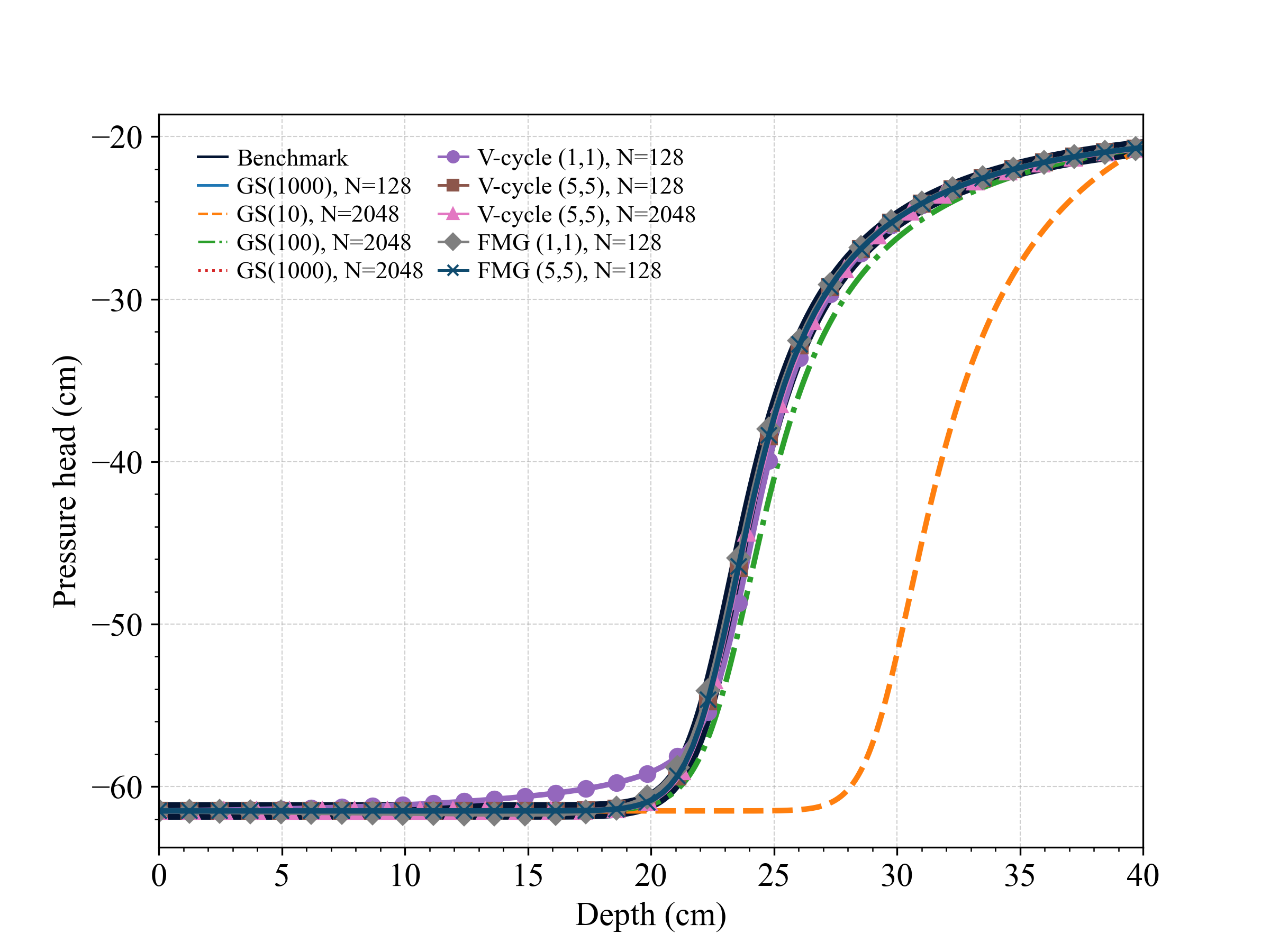}
      \includegraphics[width=0.48\textwidth]{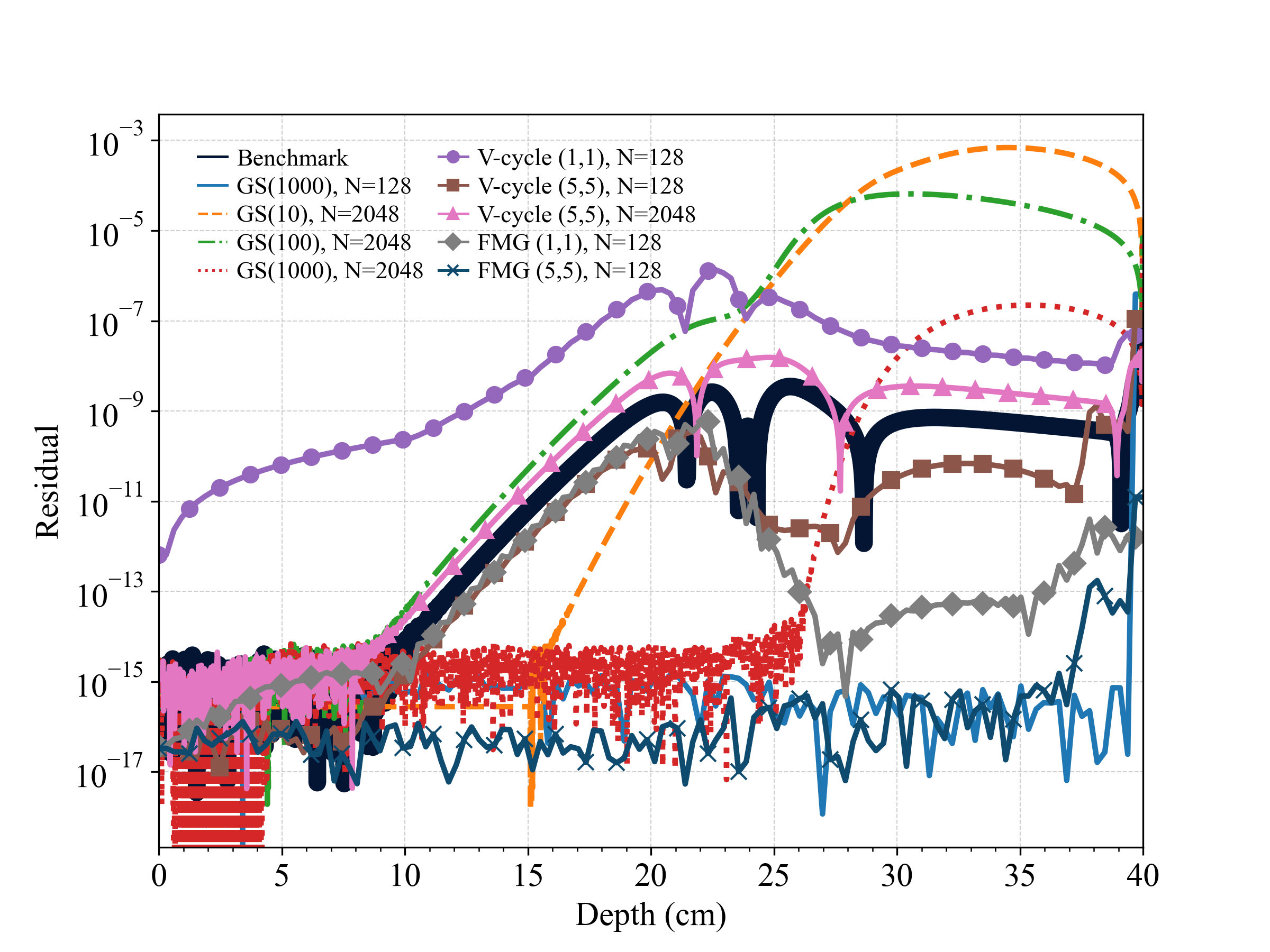}
      \includegraphics[width=0.48\textwidth]{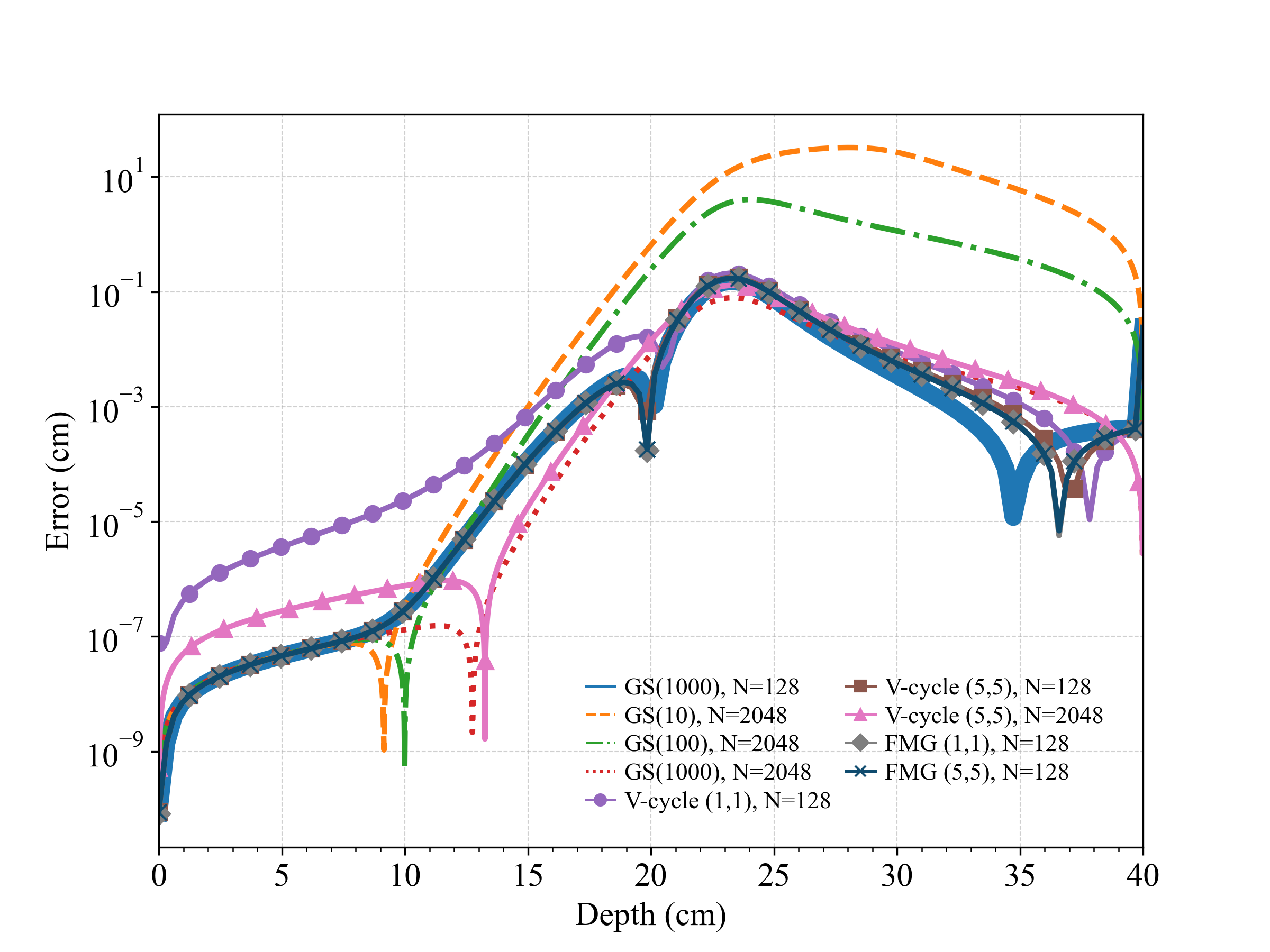}
      \includegraphics[width=0.48\textwidth]{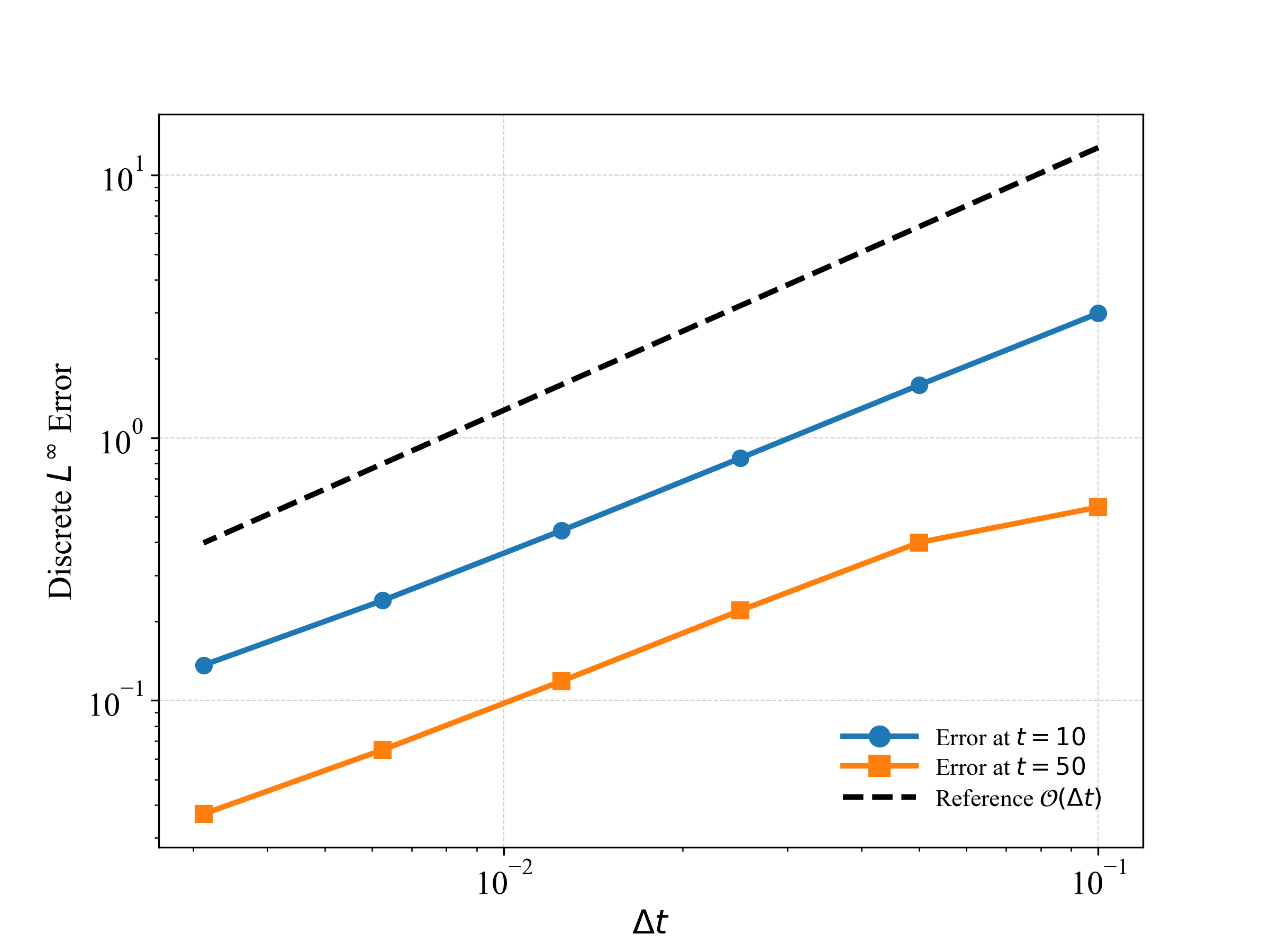}
    \end{center}
    \caption{Performance of Gauss–Seidel and multigrid solvers for
      the 1D benchmark problem. The first three panels display the
      numerical solutions, residual norms, and pointwise errors at $t=360$
      while the last panel shows the convergence behavior at $t=10$ and
      $t=50$. The first order convergence rate is confirmed.
    }\label{fig:case1d_sol_res_err}
  \end{figure}

  \begin{table}[H]
    \centering
    \caption{CPU time various methods at resolution $N
    = 2048$ up to $t = 360\mathrm{s}$}
    \label{tab:cpu_time}
    \begin{tabular}{lcccc}
      \hline
      Method   & V-cycle(5, 5) & FMG(5, 5) & GS(100) & Picard \\
      \hline
      CPU Time (s)  & 20 & 42 & 69 & 10 \\
      \hline
    \end{tabular}
  \end{table}

  The first panel of Figure.~\ref{fig:case1d_sol_res_err} presents
  the solutions computed by various solvers on two
  spatial resolutions: $N = 128$ and $N = 2048$, respectibvely. Here,
  ${\rm GS}(k)$ denotes the proposed GS iterative method with $k$ iterations,
  ${\rm V}$-${\rm cycle}(k_1, k_2)$ denotes a multigrid ${\rm V}$-${\rm cycle}$
  with $k_1$ pre-smoothing and $k_2$ post-smoothing. The benchmark solution
  is generated using a ${\rm V}$-${\rm cycle}(10, 10)$ on the $N = 2048$ grid
  with a time step $\tau =0.05$.  All other solutions are performed
  with  $\tau = 0.1$. The second and third panels of
  Figure.~\ref{fig:case1d_sol_res_err} show, respectively the residual
  and pointwise error of the numerical solutions at $t = 360$.

  The results indicate that the GS method requires a relatively
  large number of iterations to achieve convergence, whereas the
  multigrid approach attains an accurate solution within only a few
  cycles. Furthermore, the obtained numerical results are fully
  consistent with those reported in \cite{Song2026}.

  The last panel of Figure~\ref{fig:case1d_sol_res_err} presents the
  temporal convergence results of the proposed scheme at the final
  times, $t = 10$ and $t = 50$, respectively. The numerical solutions
  are computed
  on a mesh with $N = 1024$ using time steps $\tau =
  \frac{2^{-k}}{10}$, where $k = 0, 1, 2, 3, 4, 5$. These solutions
  are compared with a reference solution obtained on a finer mesh
  with $N = 2048$ and $\tau = 10^{-4}$. The errors are measured in
  the discrete $L^\infty$ norm. As expected, first-order convergence in time
  is clearly observed.

  Table~\ref{tab:cpu_time} reports the CPU time required by
  various solvers.
  As the problem size increases, one observes that both the V-cycle(5,5)
  and FMG(5, 5) schemes become substantially more efficient than the
  Gauss--Seidel method. The Picard algorithm \cite{Song2026} also
  performs well in the one-dimensional setting, owing to the
  availability of an efficient direct solver for tridiagonal linear
  systems and the simplified treatment of the nonlinear term during the
  iteration. However, such advantages do not extend to
  higher-dimensional problems, the cost of solving the
  linearized system becomes significantly greater.

  We next compare numerical solutions obtained using various
  constitutive laws. To ensure a meaningful comparison,  the parameters
  in the Gardner and van Genuchten–Mualem models are selected such that
  their hydraulic conductivity and water-retention curves resemble
  those of the Haverkamp model under the parameters specified in
  Table.~\ref{tab:case1d_haverkamp_param}.

  \begin{figure}[H]
    \begin{center}
      \includegraphics[width=0.45\textwidth]{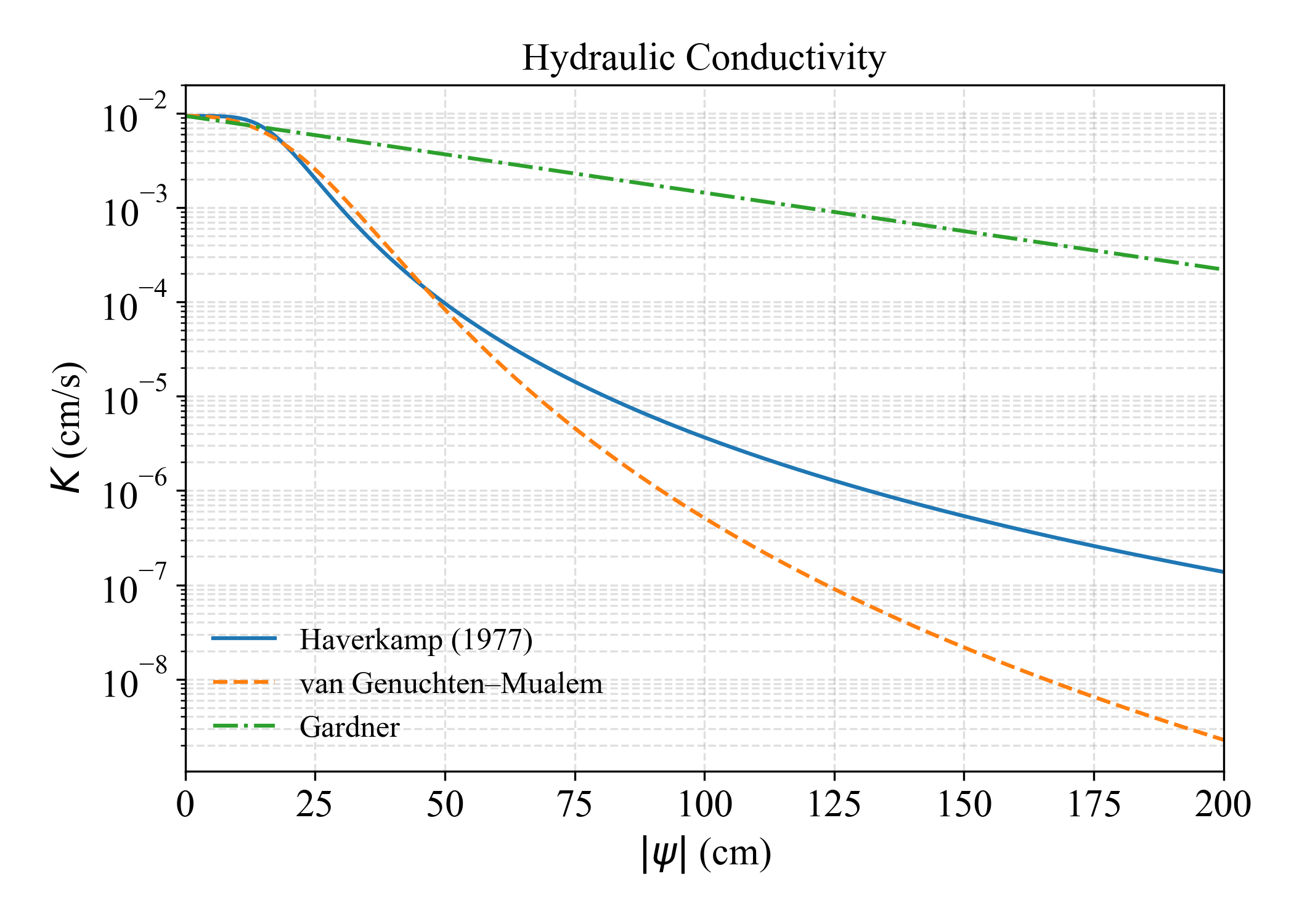}
      \includegraphics[width=0.45\textwidth]{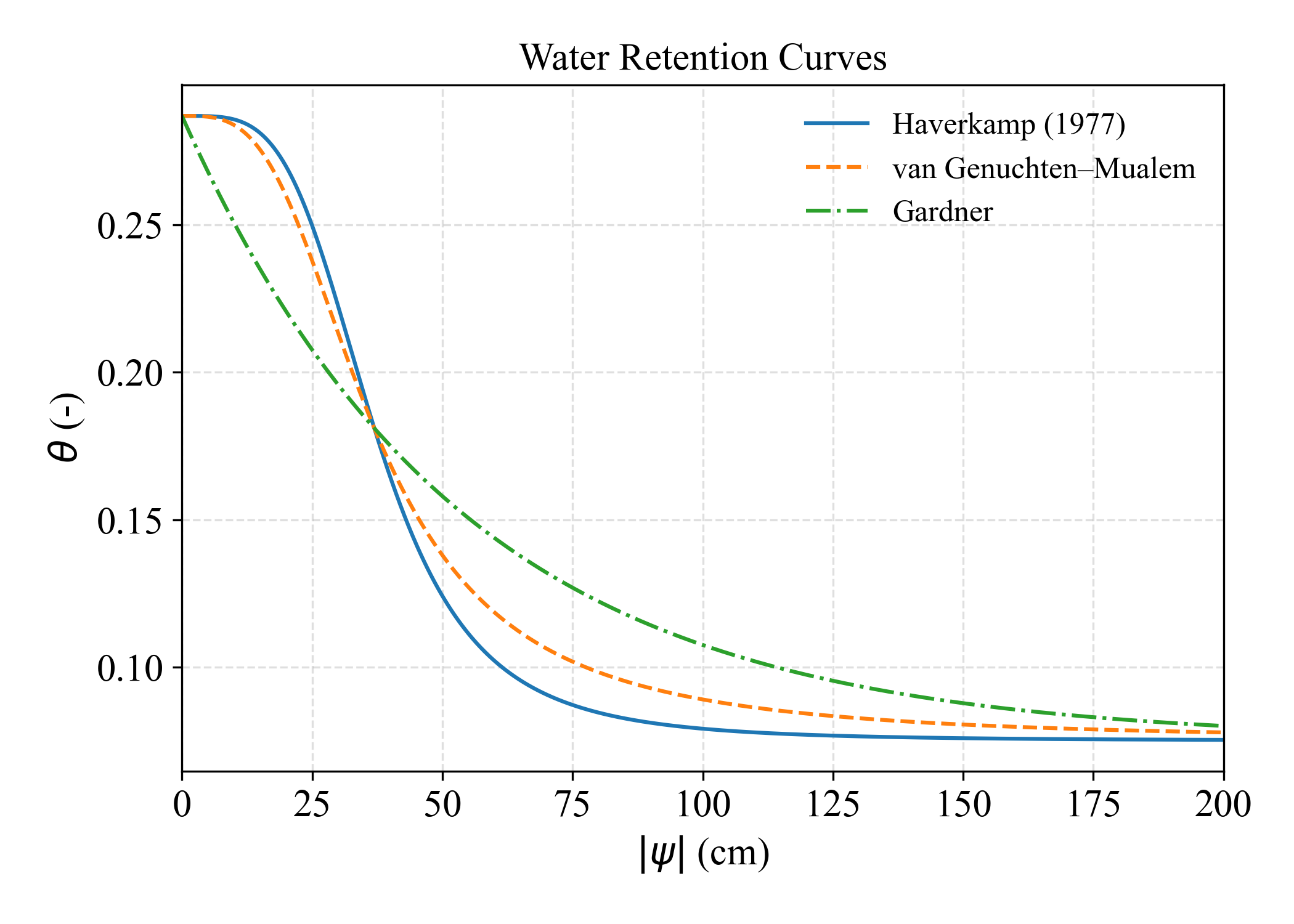}
    \end{center}
    \caption{Hydraulic conductivity and water-retention curves
      corresponding to the Haverkamp, Gardner, and van Genuchten–Mualem
    parametrizations, respectively.}\label{fig:case1d_k_theta_curve}
  \end{figure}

  Let $h \equiv |\psi|$ denote the suction head (cm). To isolate the
  effect of different constitutive relations, we keep $K_s$, $\theta_s$,
  and $\theta_r$ identical across all models. Using the Haverkamp
  parametrization, we introduce two characteristic suction levels,
  \begin{equation}
    h^{\theta}_{1/2} := \alpha^{1/\beta}
    \quad\text{and}\quad
    h^{K}_{1/2} := A^{1/\gamma},
  \end{equation}
  corresponding to $\theta = \tfrac{\theta_s+\theta_r}{2}$ and $K =
  \tfrac{K_s}{2}$, respectively. These half-suction values determine
  the horizontal placement of the WRC and HCF.

  For Gardner's exponential constitutive law
  \begin{equation}
    \theta(h)=\theta_r+(\theta_s-\theta_r)\,\mathrm{e}^{-\alpha_G h},
    \  K(h)=K_s\,\mathrm{e}^{-\alpha_G h},
  \end{equation}
  both, $\theta$ and $k$ decay at the same rate. Consequently, WRC
  and HCF share the same half-suction value. To align the WRC
  transition point with that of the
  Haverkamp model, we set
  \begin{equation}
    \alpha_G=\tfrac{\ln 2}{\,h^{\theta}_{1/2}}.
  \end{equation}

  For the van Genuchten-Mualem model, we
  let $m=1-\tfrac{1}{n}$ and $S_e=(1+(\alpha_{VG} h)^n)^{-m}$. We first
  enforce agreement in the WRC half-suction by requiring
  \begin{equation}
    S_e(h^{\theta}_{1/2})=\tfrac12
    \;\Rightarrow\;
    \alpha_{VG}(n)=\tfrac{\bigl(2^{1/m}-1\bigr)^{1/n}}{h^{\theta}_{1/2}},
  \end{equation}
  thus fixing the position of $\theta(h)$ for any chosen $n>1$.
  With $l$ taken as $0.5$, we then determine $n$ by matching
  the HCF half-suction:
  \begin{equation}
    \tfrac{K_{\text{VG}}(h^{K}_{1/2};\,\alpha_{VG}(n),\,n,\,l=0.5)}{K_s}
    \;=\;\tfrac12,
  \end{equation}
  where $K_{\text{VG}}=K_s\,S_e^{\,l}\,[\,1-(1-S_e^{1/m})^{m}\,]^2$.
  This leads to a one-dimensional root-finding problem for $n$;
  once $n^\star$ is obtained, we assign
  $\alpha_{VG}^\star=\alpha_{VG}(n^\star)$.

  Using the Haverkamp parameters  in Table.~\ref{tab:case1d_haverkamp_param},
  we obtain
  \begin{equation}
    h^{\theta}_{1/2}=\alpha^{1/\beta}\approx 36.94~\text{cm},\qquad
    h^{K}_{1/2}=A^{1/\gamma}\approx 19.08~\text{cm}.
  \end{equation}
  Hence,
  \begin{equation}
    \alpha_G=\tfrac{\ln 2}{h^{\theta}_{1/2}}
    \approx 0.01877~\text{cm}^{-1}.
  \end{equation}
  For the  van Genuchten-Mualem model, solving the above condition with
  $l=0.5$ yields
  \begin{equation}
    n^\star \approx 3.349,\
    \alpha_{VG}^\star \approx 0.03165~\text{cm}^{-1}, \
    m=1 - \tfrac{1}{n^\star}\approx 0.701.
  \end{equation}
  These choices align the Gardner and van Genuchten curves with the
  Haverkamp model at the key transition points while preserving the
  characteristic tail behaviors of each constitutive law, as
  illustrated in Figure~\ref{fig:case1d_k_theta_curve}.

  \begin{table}[H]
    \centering
    \begin{tabular}{lccc}
      \toprule
      \textbf{Symbol} & \textbf{Gardner} & \textbf{van
      Genuchten--Mualem} & \textbf{Unit} \\
      \midrule
      $K_s$              & $0.00944$                & $0.00944$
      & cm/s \\
      $\theta_s$         & $0.287$              & $0.287$           & -- \\
      $\theta_r$         & $0.075$              & $0.075$           & -- \\
      $\alpha$         & $0.01877$            & $0.03165$
      & cm$^{-1}$ \\
      $n$                & --                   & $3.349$              & -- \\
      $m$                & --                   & $0.701$              & -- \\
      $l$                & --                   & $0.5$                & -- \\
      \bottomrule
    \end{tabular}
    \caption{Parameters used in the Gardner and van Genuchten--Mualem
    constitutive laws.}
    \label{tab:constitutive_param}
  \end{table}

  \begin{figure}
    \begin{center}
      \includegraphics[width=0.24\textwidth]{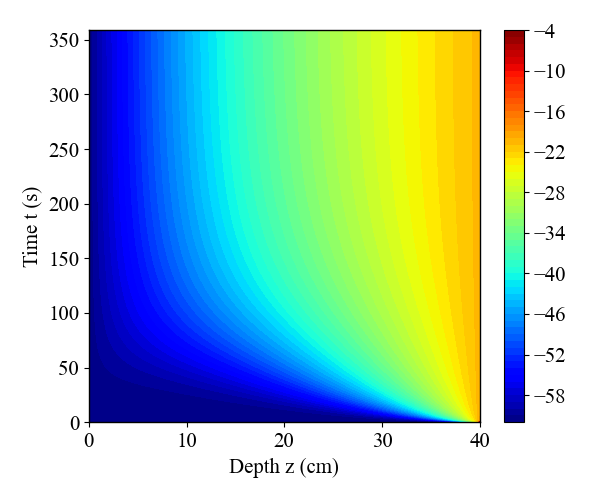}
      \includegraphics[width=0.24\textwidth]{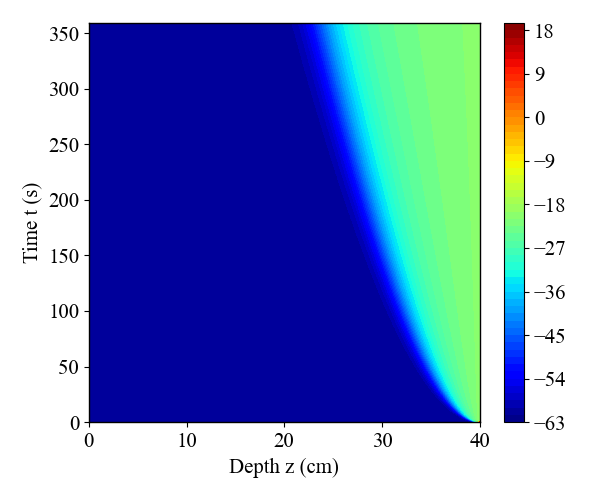}
      \includegraphics[width=0.24\textwidth]{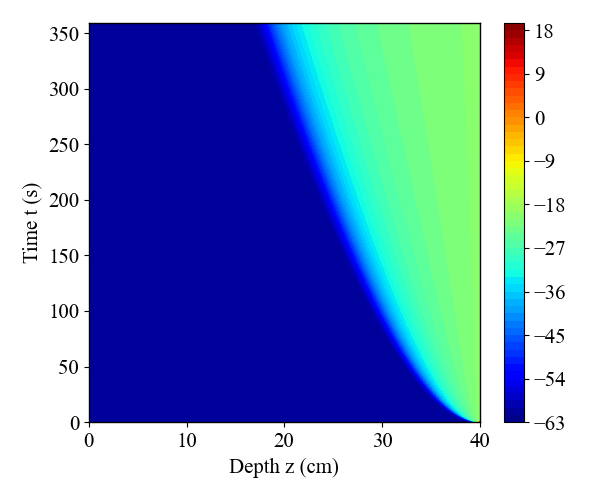}
      \includegraphics[width=0.24\textwidth]{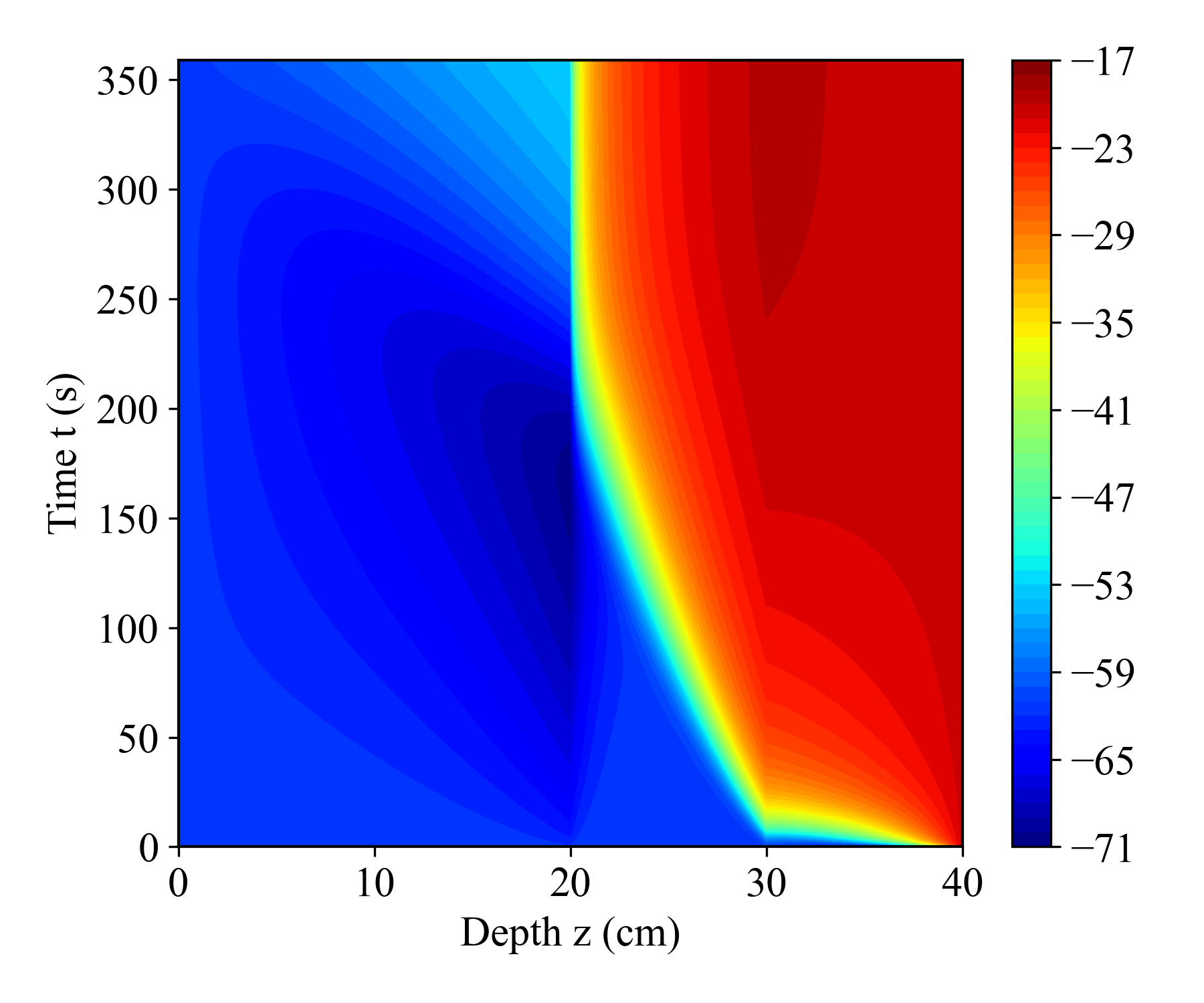}
    \end{center}
    \caption{Comparison of solutions from $t = 0$ to $t =
      360$ obtained using different constitutive models.
      From left to right: Gardner model, Haverkamp model, van
    Genuchten model, and the composite model.}\label{fig:solution_diff_k_theta}
  \end{figure}

  Instead of considering the three constitutive laws separately, we
  also examine a test case in which the constitutive relations are
  defined as a spatially composite function of the three models.
  Specifically, the WRC and HCF are given by
  \begin{equation}
    \theta(\psi, z) =
    \left\lbrace
    \begin{aligned}
      & \theta_{Gardner}(\psi), & 30 \leq z < 40, \\
      & \theta_{Haver}(\psi), & 20 \leq z < 30, \\
      & \theta_{VG}(\psi), &\text{else}, \\
    \end{aligned}
    \right., \quad
    K(\psi, z) =
    \left\lbrace
    \begin{aligned}
      & K_{Gardner}(\psi), & 30 \leq z < 40, \\
      & K_{Haver}(\psi), & 20 \leq z < 30, \\
      & K_{VG}(\psi), &\text{else}, \\
    \end{aligned}
    \right.
  \end{equation}

  Figure~\ref{fig:solution_diff_k_theta} presents four contour plots
  showing the temporal evolution of the pressure head obtained using
  the Gardner, Haverkamp, van Genuchten–Mualem, and composite
  constitutive laws, respectively. The solution corresponding to the
  Gardner model (left panel) exhibits a markedly different infiltration
  behavior: the wetting front propagates more rapidly and features a
  smoother and more diffuse transition. This behavior is consistent
  with the exponential form of the hydraulic conductivity and soil
  water retention functions inherent in the Gardner model.

  In contrast, the Haverkamp and van Genuchten–Mualem solutions are
  remarkably similar. Their wetting fronts exhibit almost identical
  contour shapes. This close resemblance arises
  from the fact that their HCF and WRC curves were calibrated to match
  at key transition points, resulting in constitutive behaviors that
  differ only slightly in phase but not in overall structure.
  Consequently, the two models produce nearly indistinguishable
  transient profiles over the full time interval, $0 \leq t \leq 360$.

  The solution obtained with the composite constitutive law exhibits a
  piecewise infiltration dynamics that clearly reflects the layered
  definition of the soil hydraulic properties. Distinct changes in the
  propagation speed and shape of the wetting front can be observed at
  the interfaces $z=20$ and $z=30$, where the constitutive relations
  switch between the van Genuchten–Mualem, Haverkamp, and Gardner
  models. In particular, the wetting front accelerates within the
  Gardner layer, while a comparatively sharper front is maintained in
  the Haverkamp and van Genuchten–Mualem regions. Despite the
  discontinuities in the constitutive functions, the pressure head
  field remains continuous across the interfaces, indicating that the
  proposed numerical scheme robustly captures heterogeneous soil
  profiles with spatially varying hydraulic characteristics.
  \subsection{2-d problem}

  We consider a two-dimensional infiltration test in the domain, $(0,
  10)^2$. The soil hydraulic properties are given by the Gardner
  constitutive relation in Table~\ref{tab:models}. The
  initial pressure head is $h(\mathbf{x}, 0 ) = -10$. At the top
  boundary, we prescribe a nonuniform Dirichlet
  condition
  \begin{equation}\label{eq:top-bc}
    h(x,10,t) = \frac{1}{\alpha}
    \ln\!\Bigl(
      e^{-10 \alpha}
      + \left(1 - e^{-10\alpha}\right)
      \sin\frac{\pi x}{10}
    \Bigr),
  \end{equation}
  which yields \(h = 0\) at the midpoint \(x = 5\) and smoothly
  decreases to \(h = -10\) at the lateral edges \(x = 0\) and
  \(x = 10\). On all other boundaries, we impose a constant head $h = -10$.
  In the simulations we set
  \(\theta_s = 0.45\),
  \(\theta_r = 0.15\),
  \(\alpha = 0.164\mathrm{m}^{-1}\),
  \(K_s = 0.1\mathrm{m} \cdot \mathrm{day}^{-1}\),
  The problem is simulated over a 3-day interval. We use a
  spatial resolution of $N = 128$ with a time step size $\tau = 0.1$.
  \begin{figure}[H]
    \begin{center}
      \includegraphics[width=0.5\textwidth]{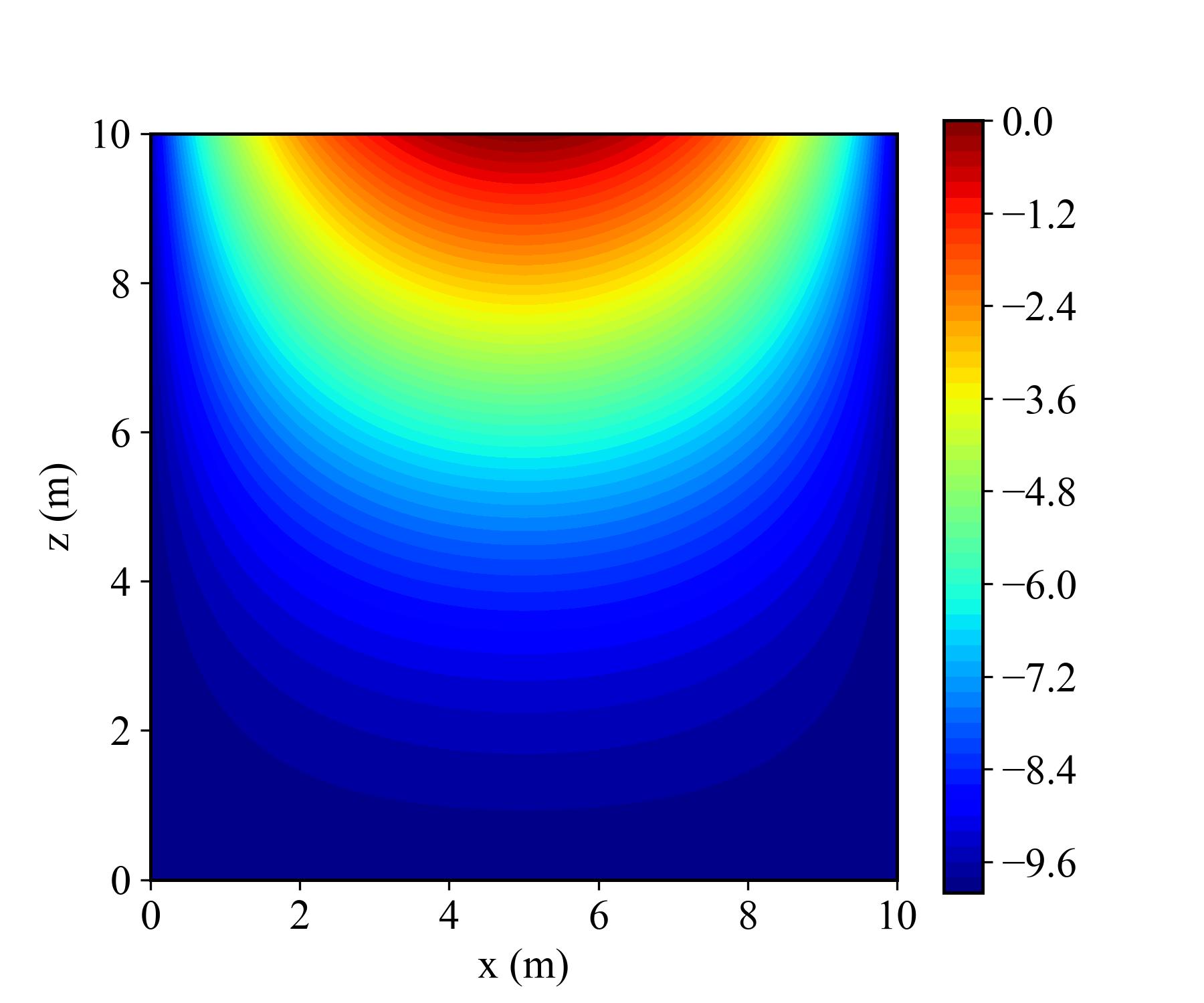}
    \end{center}
    \caption{Pressure head distribution at $t = 3$ days for the 2D
    infiltration test.}
    \label{fig:2d-exp-snapshot}
  \end{figure}

  A snapshot of the numerical solution at $t = 3$ days is shown in
  Figure~\ref{fig:2d-exp-snapshot}. The shape and position of the
  wetting front are in good agreement with the reference results
  reported in \cite{Chen2023}. In addition, the total CPU time
  required for this example is $2.216\mathrm{s}$.

  We next assess the performance of the proposed algorithms for
  simulating infiltration in a heterogeneous medium.
  A two-layer soil configuration is considered, and the van
  Genuchten–Mualem constitutive law is adopted.
  The computational domain is chosen as $\Omega = (0, 100)^2$, and the
  two soil layers are separated by a smooth interface (see
  Figure~\ref{fig:two-layer})
  \begin{equation}
    \zeta(x) = 100 \left( 0.1 \left(1 - \cos{\left(\tfrac{\pi
    x}{100}\right)}\right) + 0.45\right).
  \end{equation}

  The hydraulic parameters in the van Genuchten--Mualem constitutive
  law are specified as follows:
  \begin{equation}
    \begin{aligned}
      & z \geq \zeta(x): \quad \theta_s =  0.50, \ \theta_r = 0.120, \ \alpha
      = 0.025\mathrm{cm}^{-1}, \ n = 3.00, \ K_s =
      0.25\mathrm{cm}\cdot h^{-1}, \\
      & z < \zeta(x): \quad \theta_s =  0.46, \ \theta_r = 0.034, \ \alpha
      = 0.016\mathrm{cm}^{-1}, \ n = 1.37, \ K_s =
      2.00\mathrm{cm}\cdot h^{-1}, \\
    \end{aligned}
  \end{equation}
  The initial condition is prescribed as $\psi(x, z, 0) = -z$, and we
  impose homogeneous Dirichlet boundary conditions on the top and
  bottom boundaries, while homogeneous Neumann conditions are applied
  on the left and right boundaries.

  A uniform grid of size $256 \times 256$ is used in space, and the time step
  is set to $\tau = 0.01$. To solve the nonlinear system, the FMG
  method is employed with 10 pre-smoothing and 10 post-smoothing steps.
  In the nonlinear Gauss-Seidel smoother, we take $\bm{\Lambda}_\kappa
  = 0.1\mathbf{I}$, $\xi=0$.

  \begin{figure}[htbp]
    \centering
    \begin{tikzpicture}[scale=0.03]

      \def\interface{
        (0,45)
        (20,47)
        (40,52)
        (60,57)
        (80,61)
        (100,63)
      }

      \def\interfaceR{
        (100,63)
        (80,61)
        (60,57)
        (40,52)
        (20,47)
        (0,45)
      }

      \path[fill=blue!8]
      (0,100) -- (100,100) --       
      (100,63) --                   
      plot[smooth] coordinates \interfaceR --  
      (0,100) --                    
      cycle;

      \path[fill=brown!10]
      (0,0) -- (100,0) --            
      (100,63) --                    
      plot[smooth] coordinates \interfaceR --  
      (0,0) --                       
      cycle;

      \draw[line width=0.8pt] (0,0) rectangle (100,100);

      \draw[line width=1.2pt]
      plot[smooth] coordinates \interface;

      \node[below] at (0,0) {0};
      \node[below] at (100,0) {100};
      \node[left]  at (0,100) {100};

      \node at (55,75) {$z \ge \zeta(x)$};
      \node at (55,28) {$z \le \zeta(x)$};

    \end{tikzpicture}
    \caption{Geometry of the two-layer heterogeneous soil domain, where
    the layers are separated by a smooth interface: $z = \zeta(x)$.}
    \label{fig:two-layer}
  \end{figure}
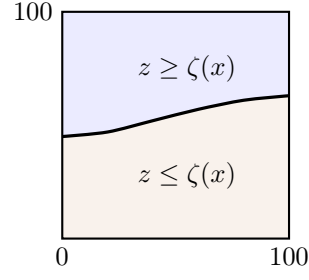

  \begin{figure}[H]
    \begin{center}
      \includegraphics[width=0.24\textwidth]{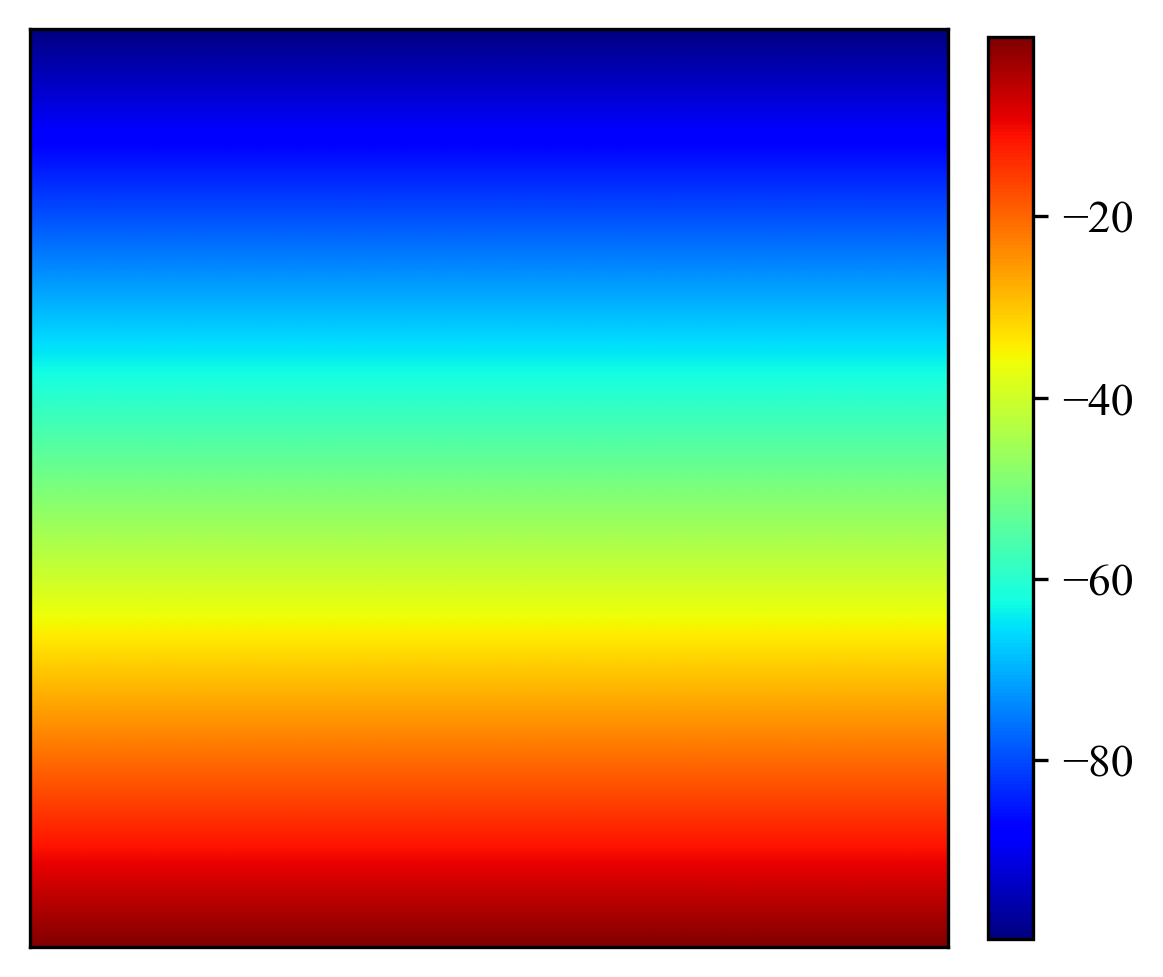}
      \includegraphics[width=0.24\textwidth]{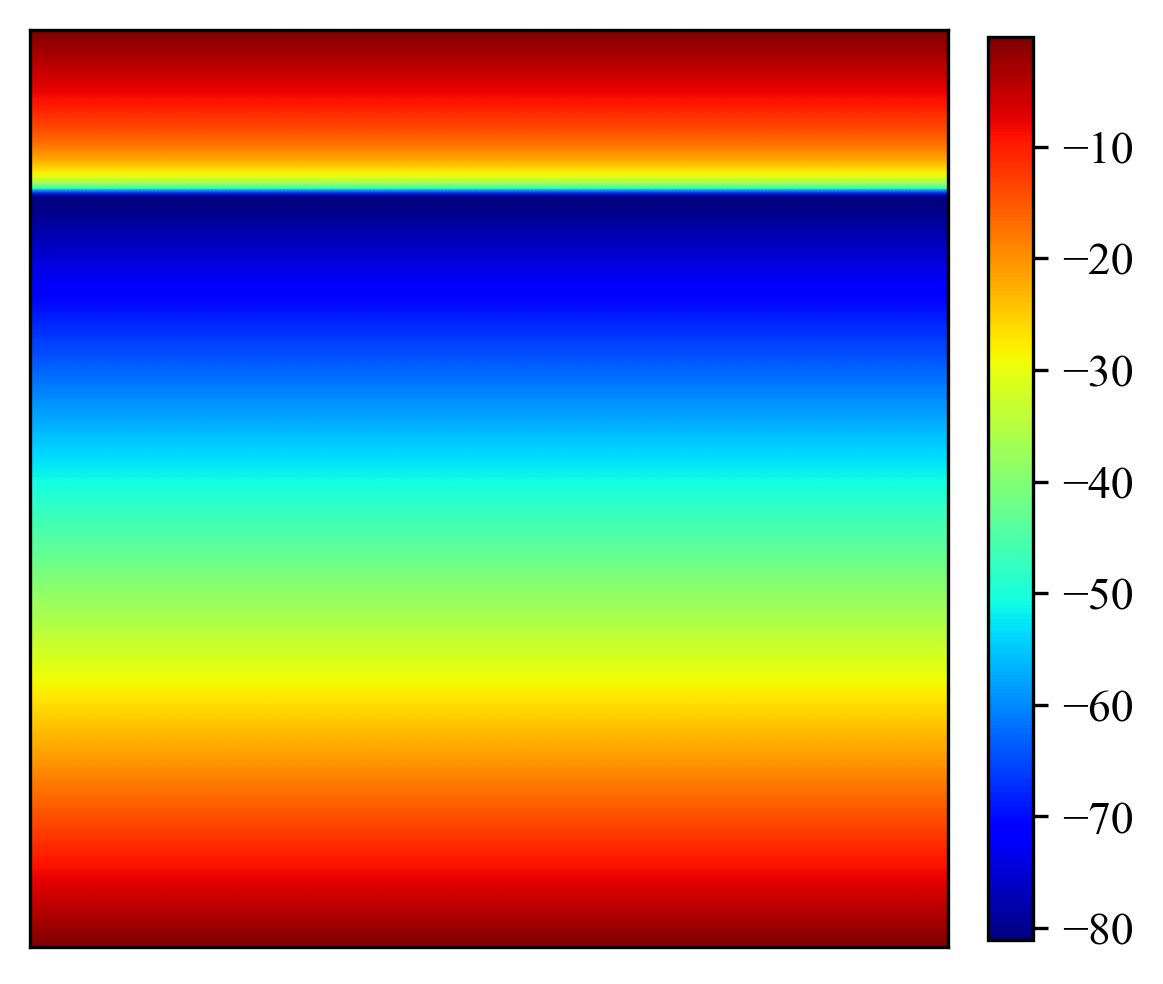}
      \includegraphics[width=0.24\textwidth]{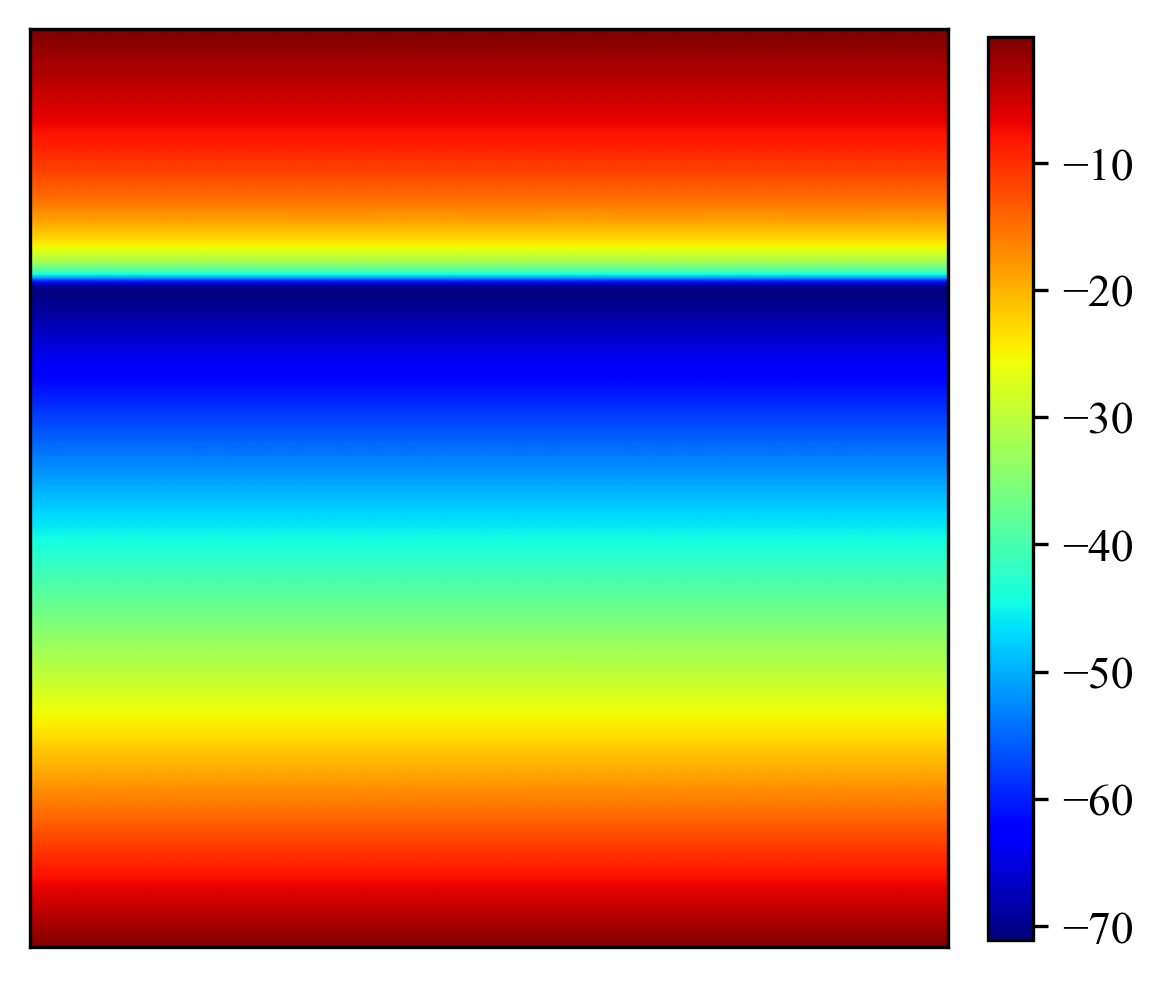}
      \includegraphics[width=0.24\textwidth]{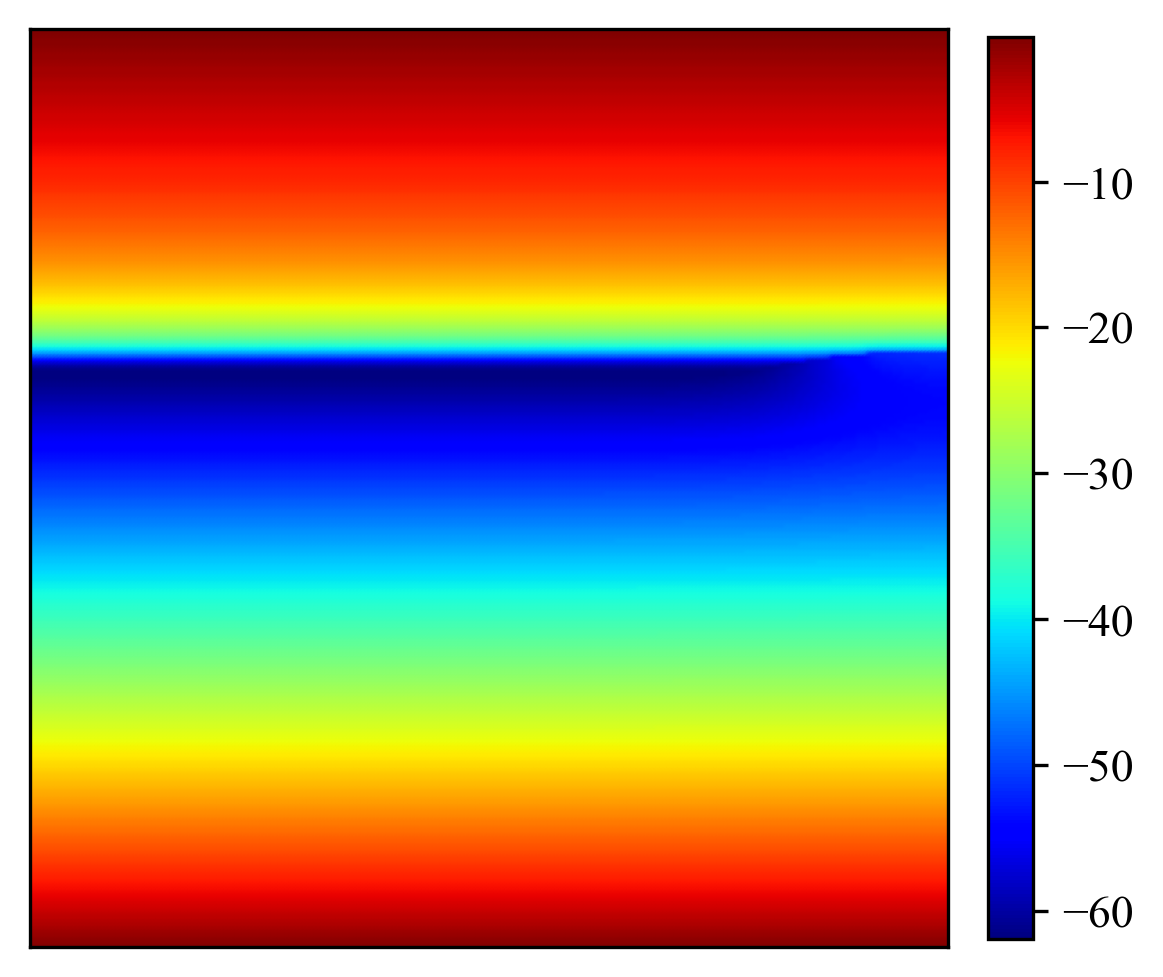}
      \includegraphics[width=0.24\textwidth]{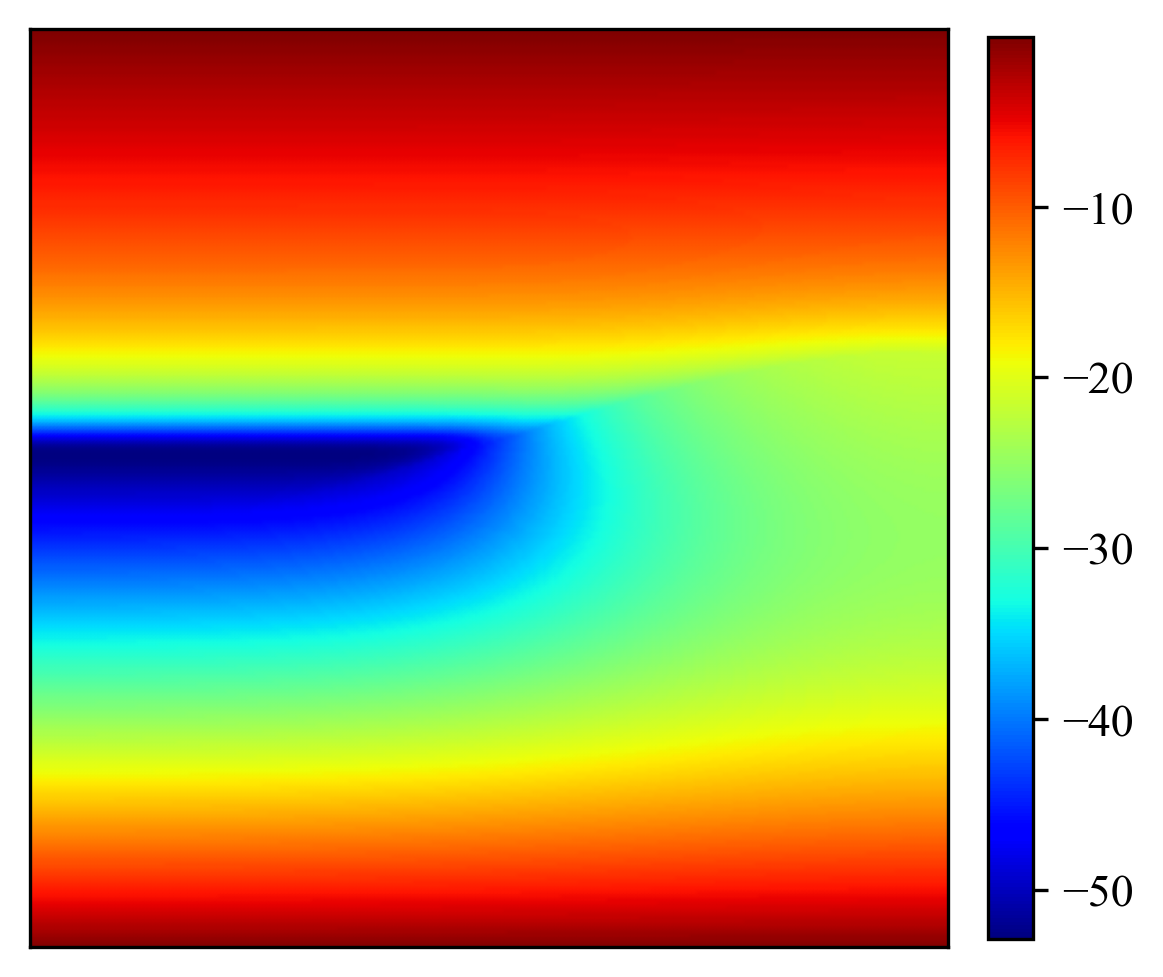}
      \includegraphics[width=0.24\textwidth]{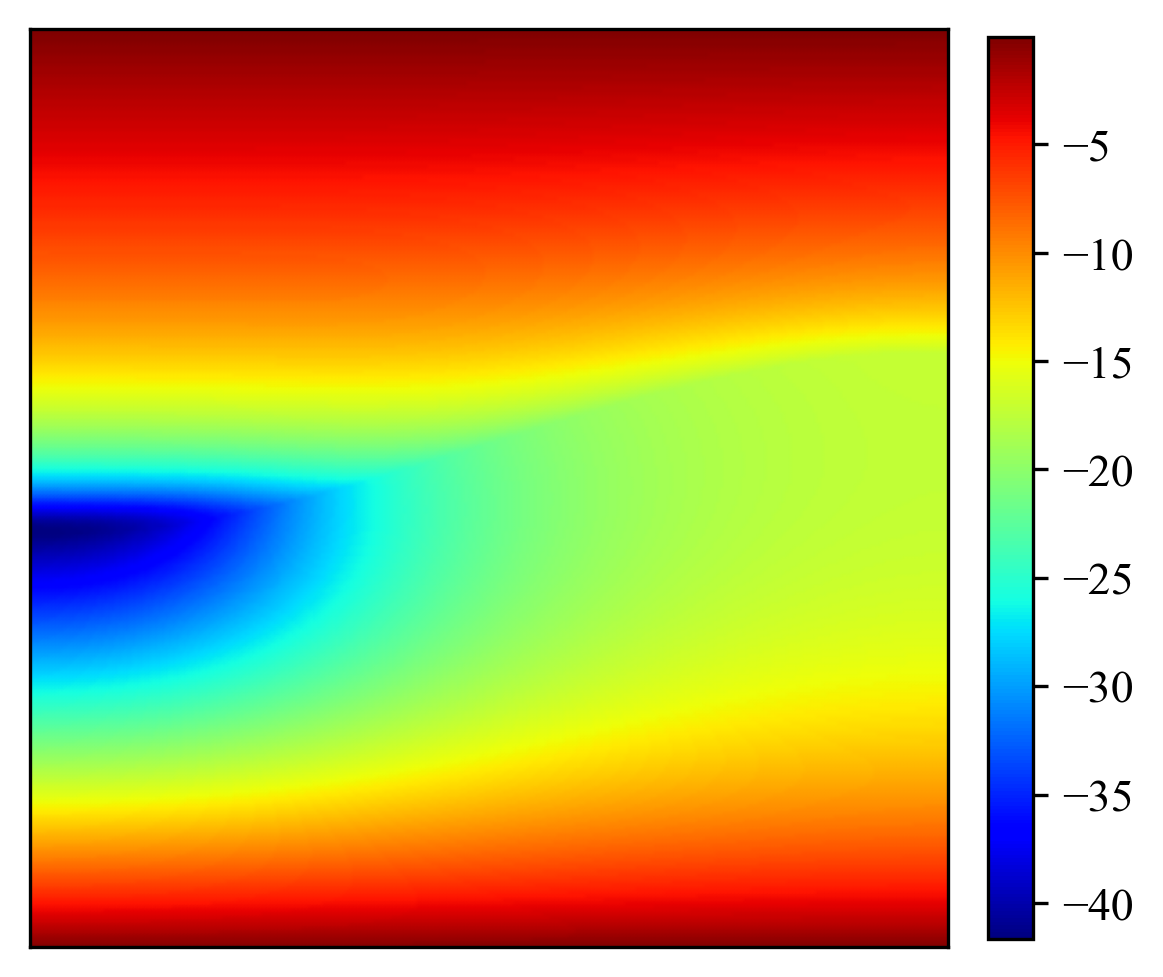}
      \includegraphics[width=0.24\textwidth]{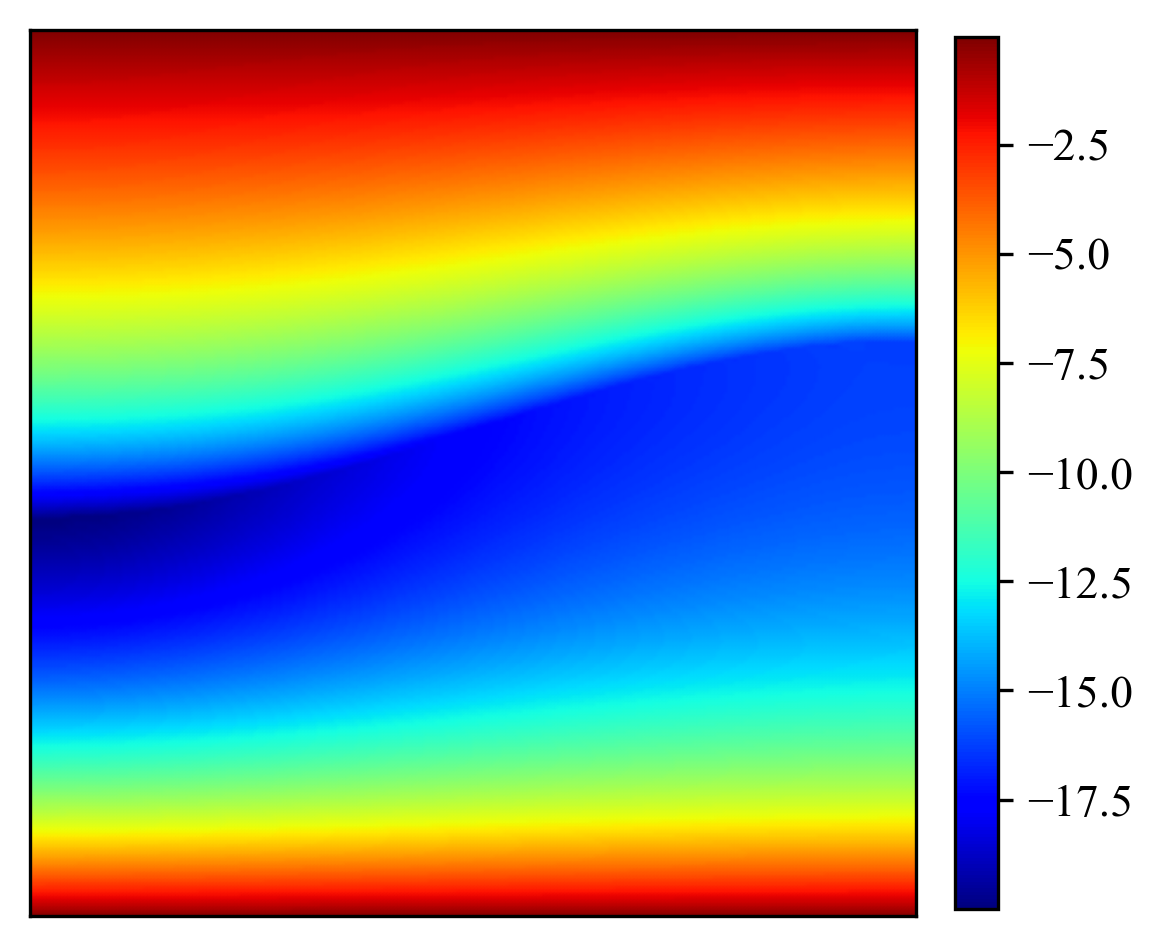}
      \includegraphics[width=0.24\textwidth]{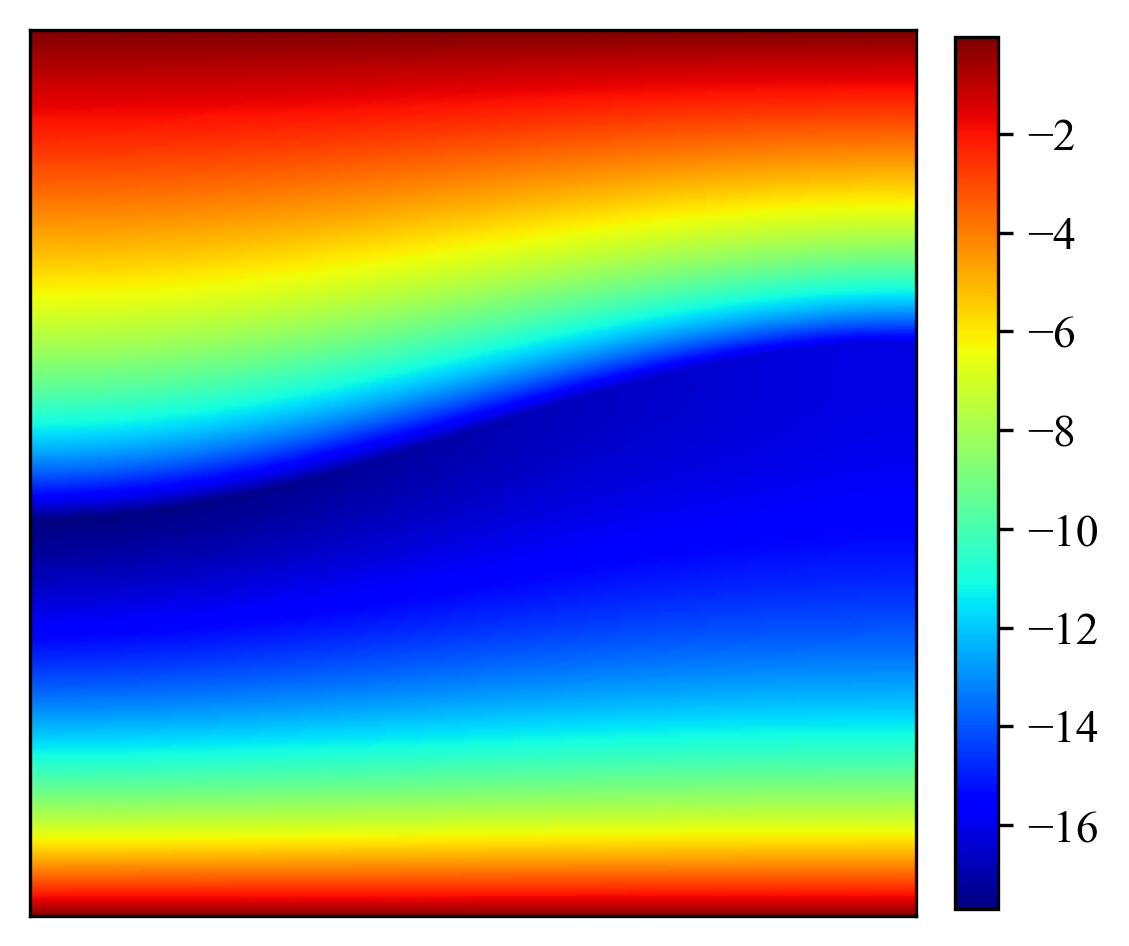}
    \end{center}
    \caption{Snapshots of the pressure head distribution at
      $t =
      0,\ 0.25,\ 0.5,\ 0.75,\ 1.0,\ 1.25,\ 1.5,\ 1.75\ \text{days}$,
      respectively,
      shown from left to right and top to bottom.
    }\label{fig:ex2d_02_pressure_head}
  \end{figure}

  \begin{figure}[H]
    \begin{center}
      \includegraphics[width=0.24\textwidth]{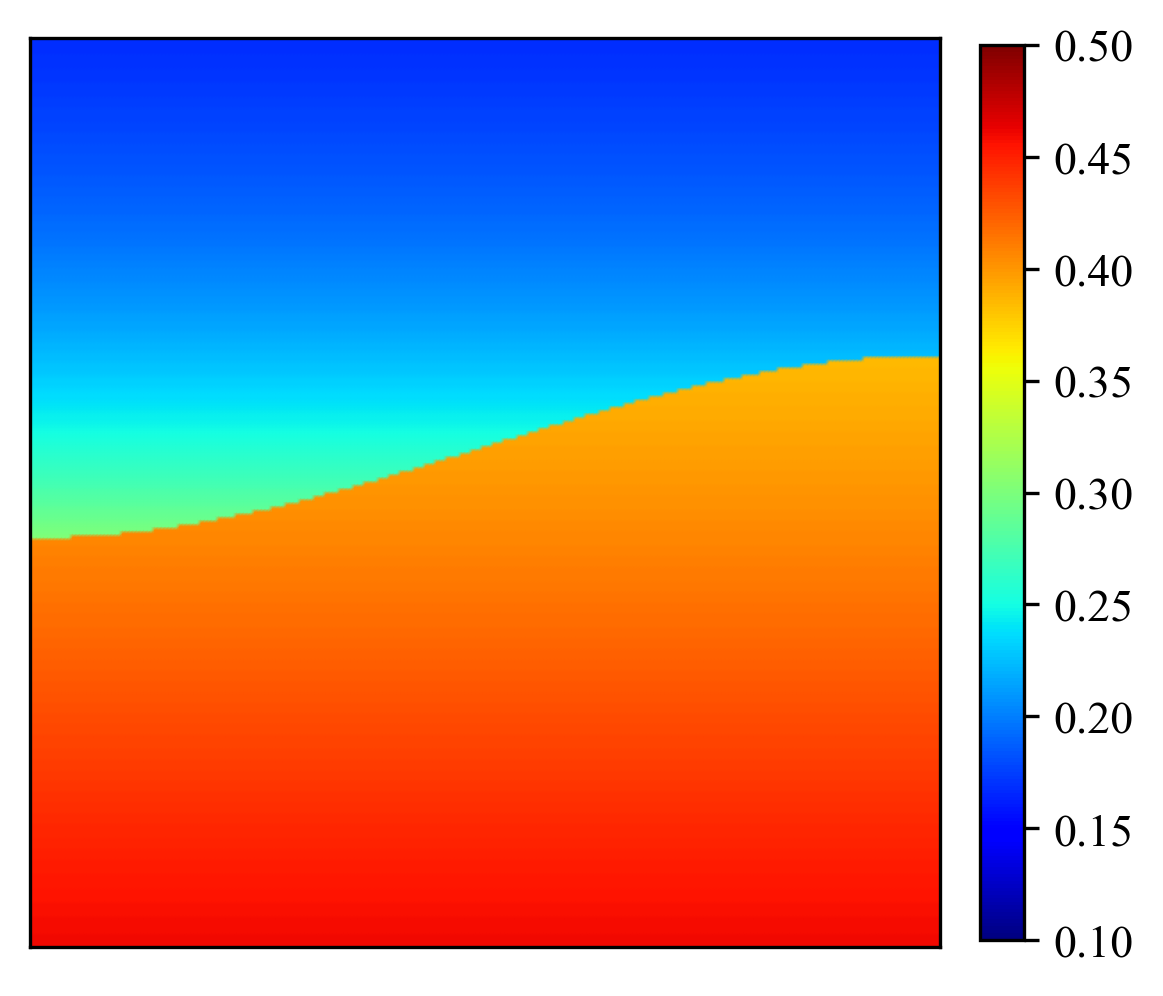}
      \includegraphics[width=0.24\textwidth]{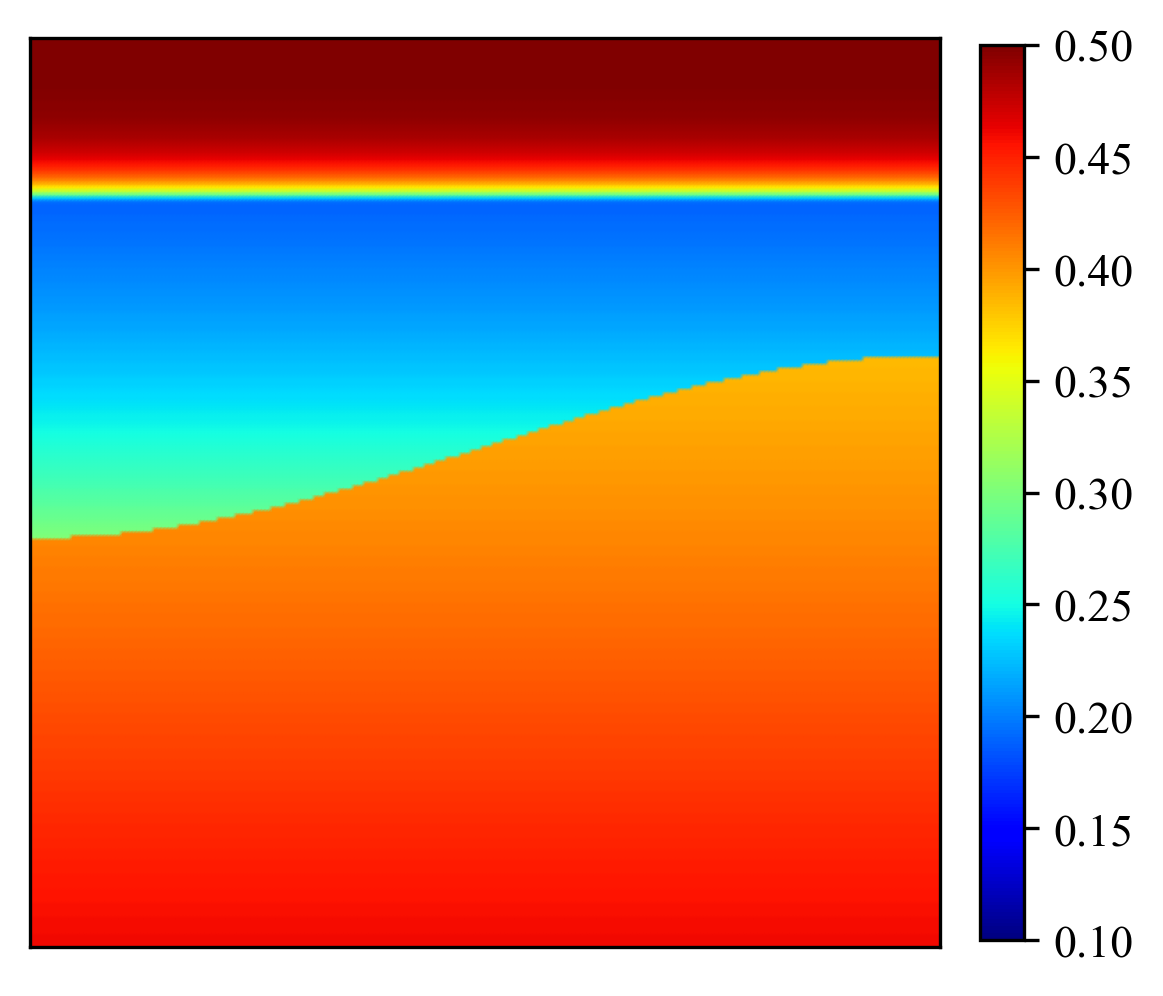}
      \includegraphics[width=0.24\textwidth]{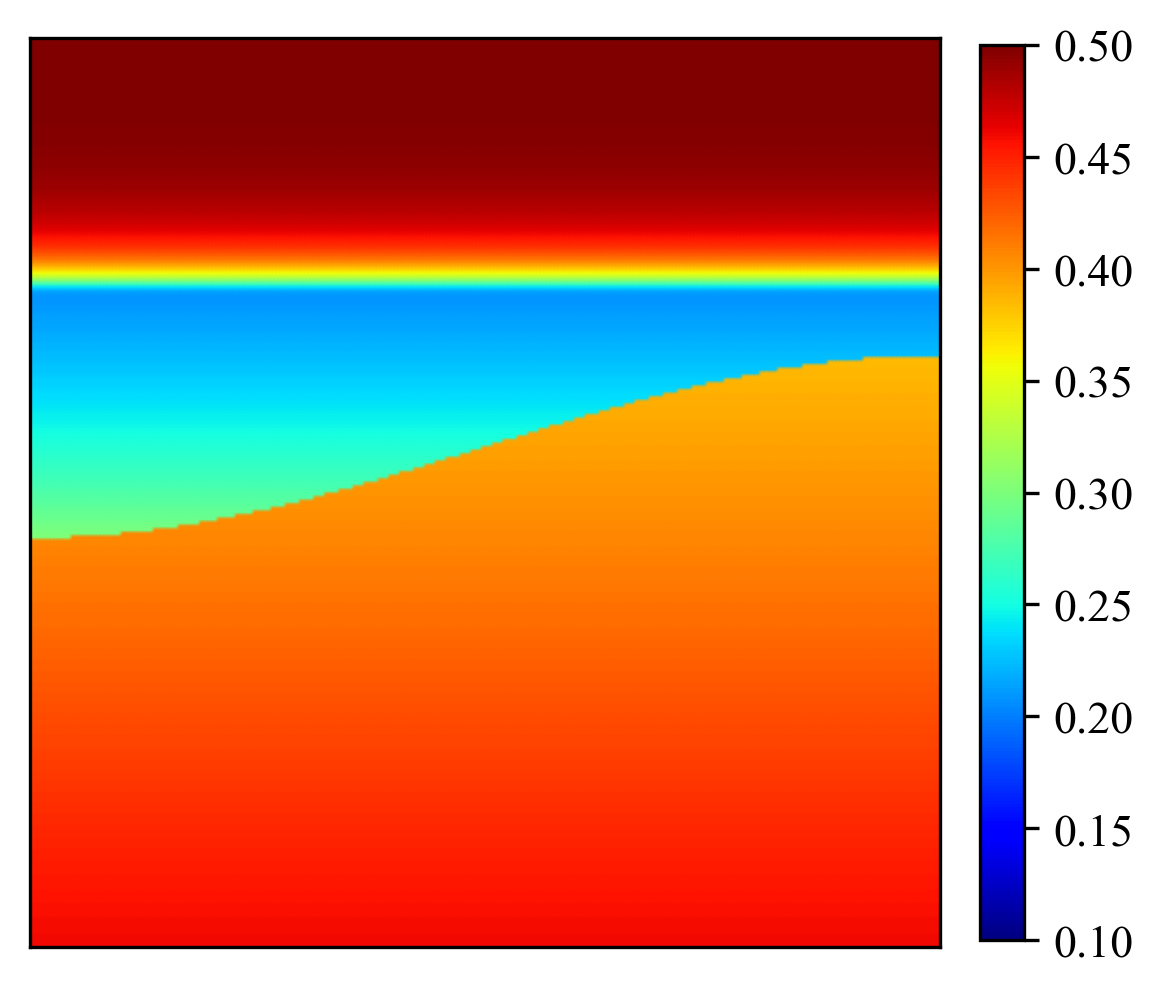}
      \includegraphics[width=0.24\textwidth]{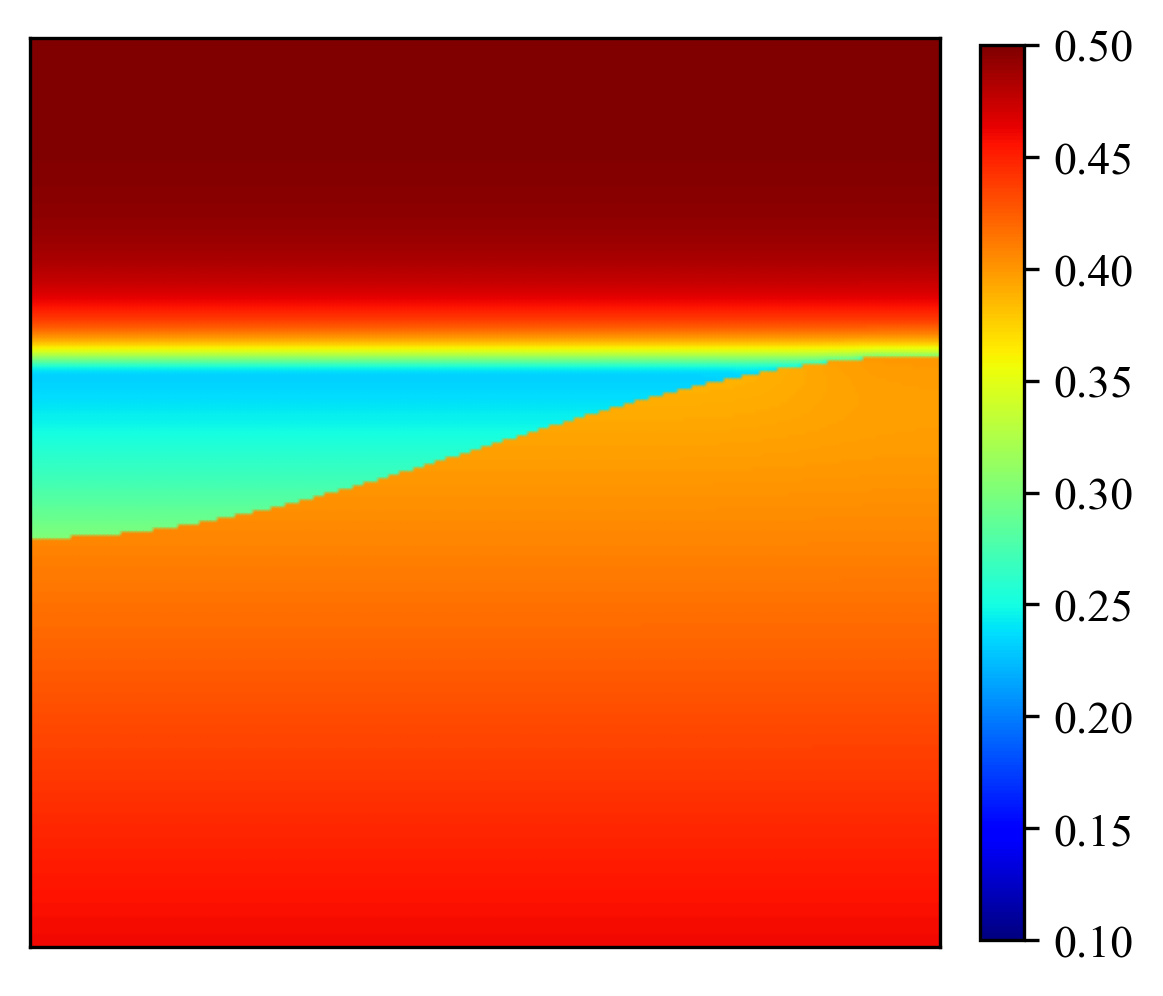}
      \includegraphics[width=0.24\textwidth]{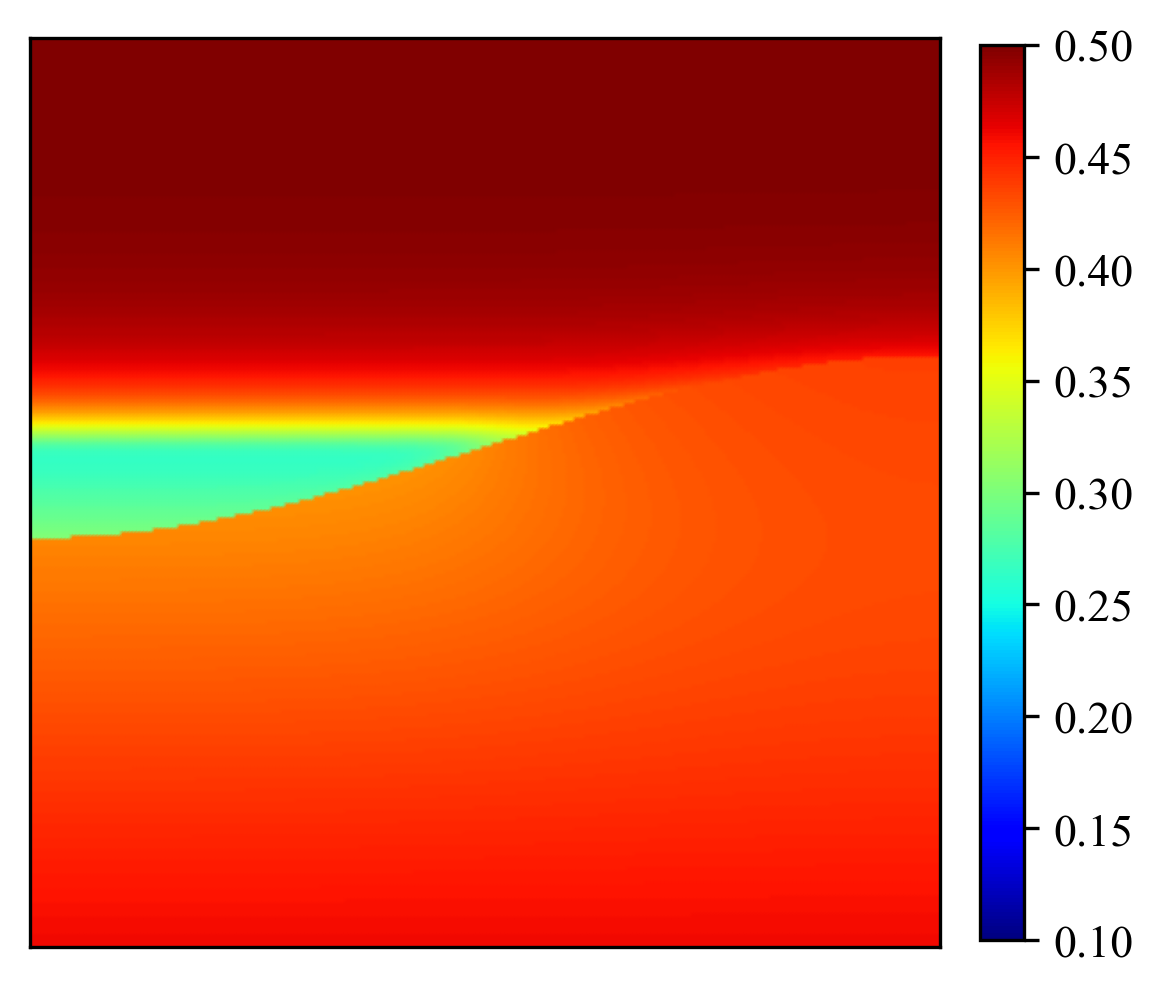}
      \includegraphics[width=0.24\textwidth]{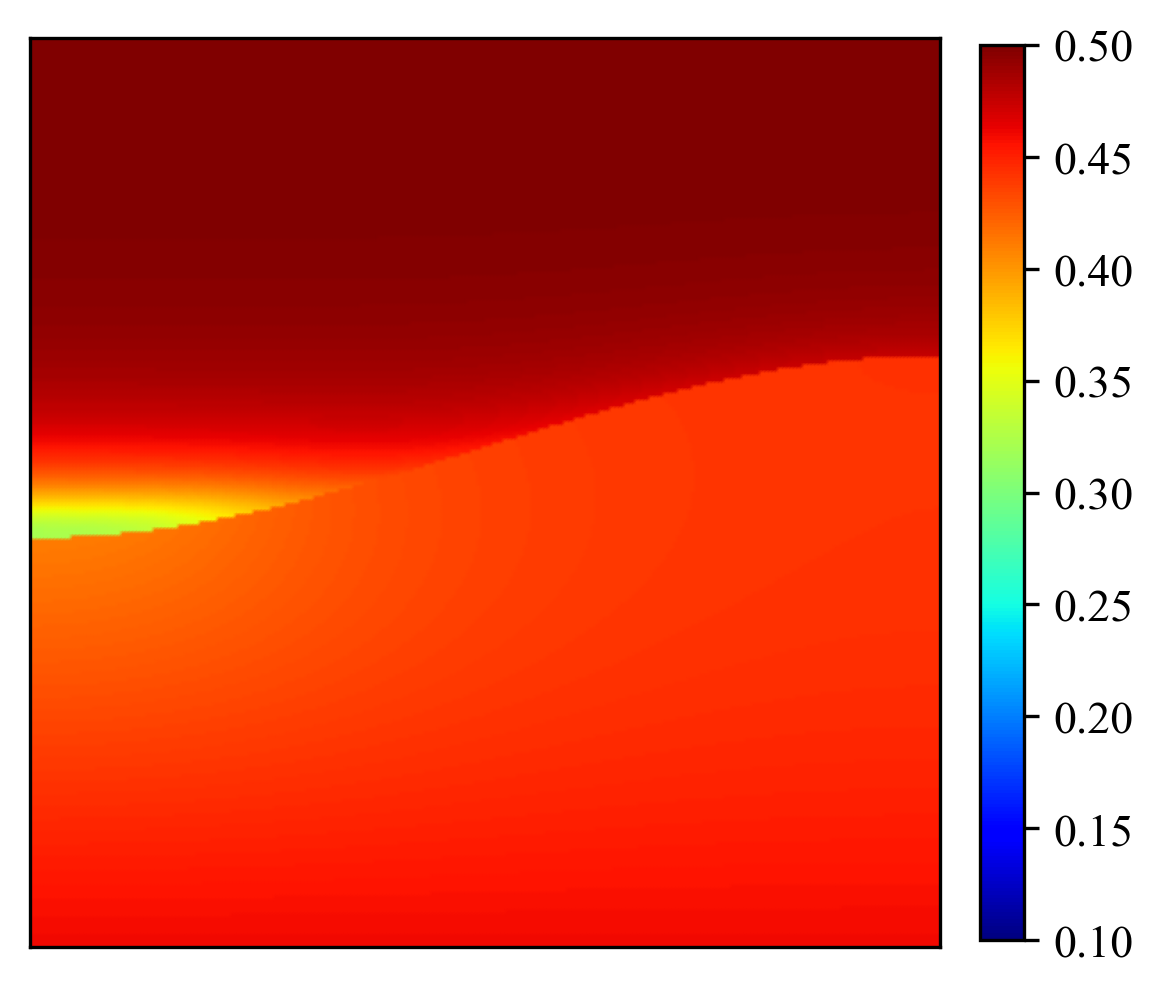}
      \includegraphics[width=0.24\textwidth]{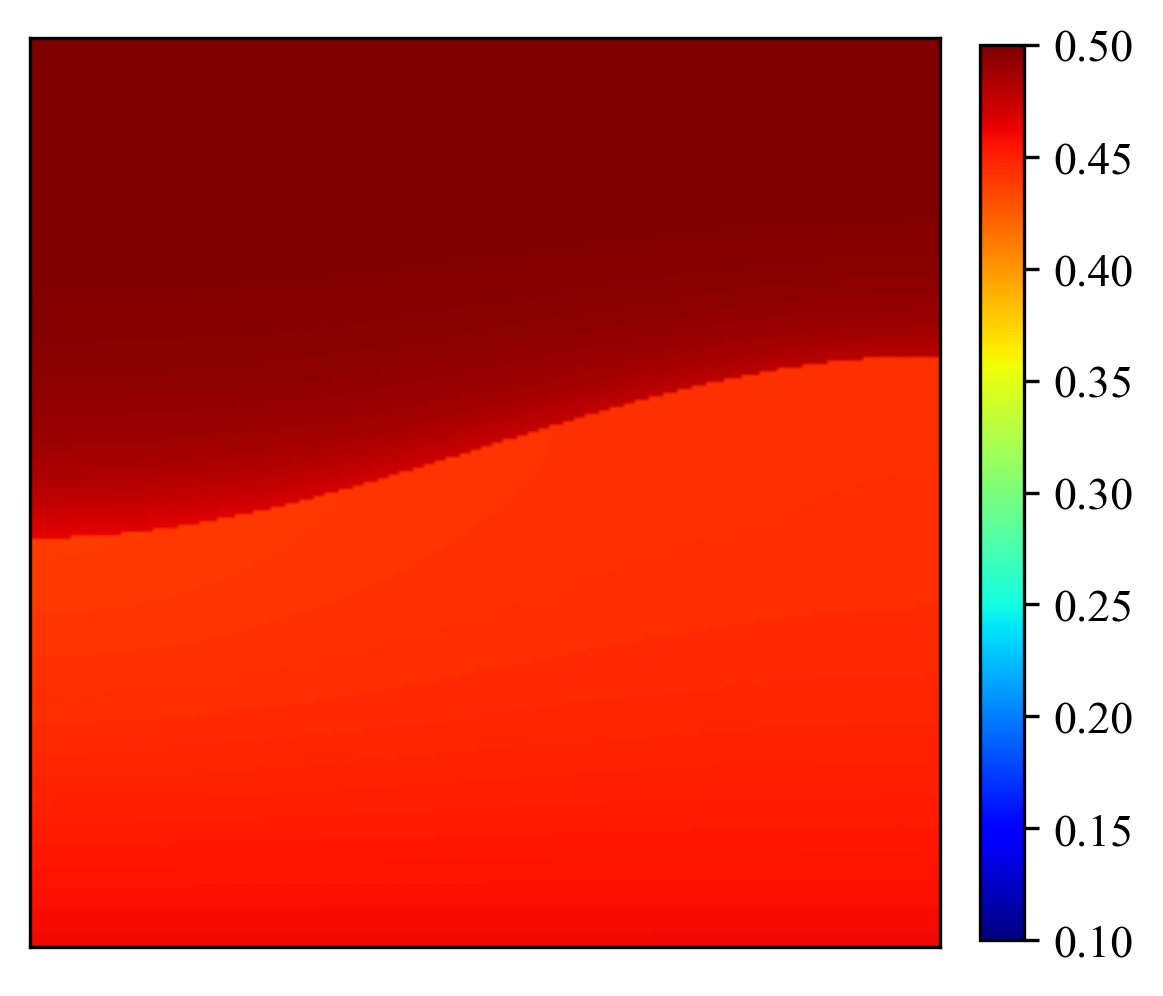}
      \includegraphics[width=0.24\textwidth]{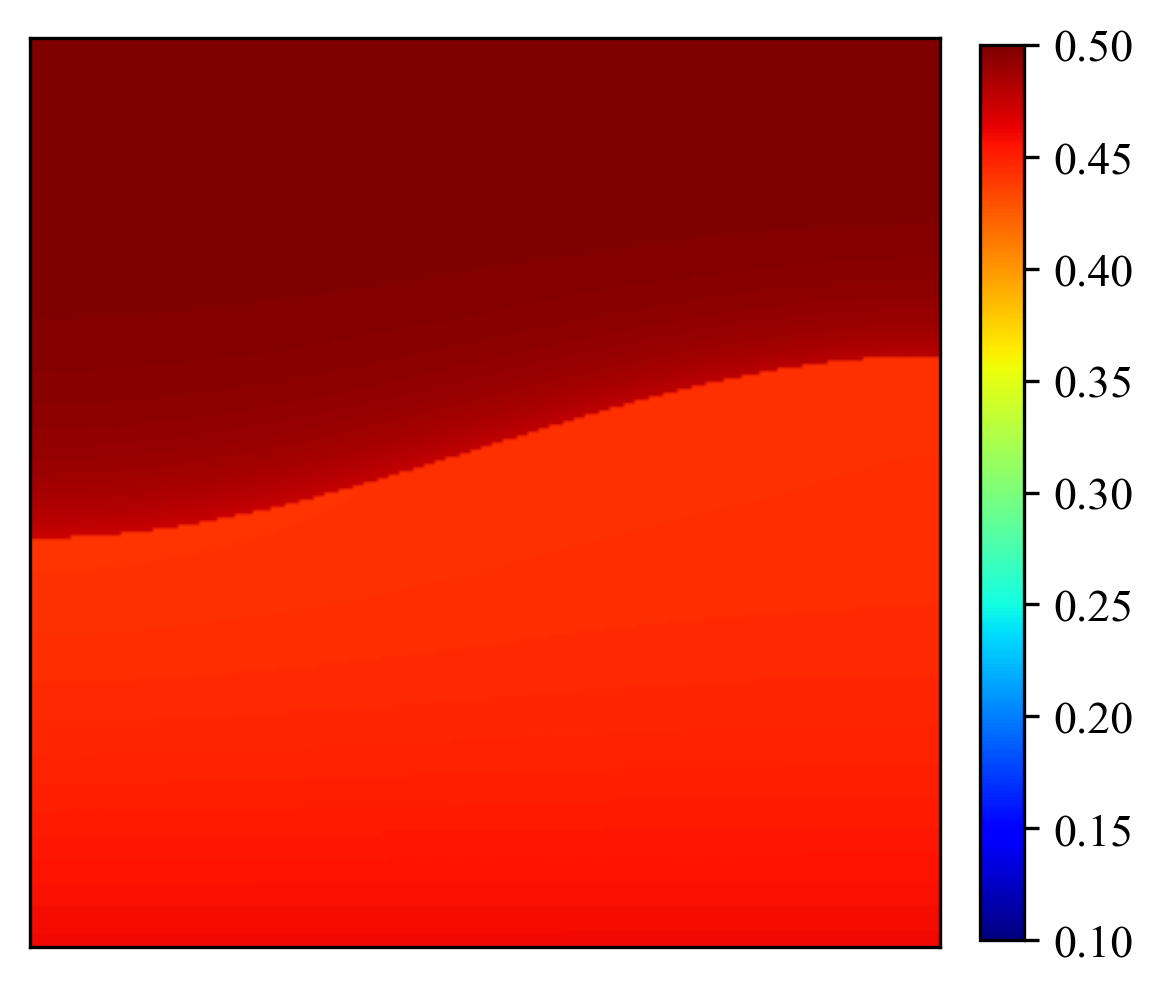}
    \end{center}
    \caption{Snapshots of the WRC at
      $t =
      0,\ 0.25,\ 0.5,\ 0.75,\ 1.0,\ 1.25,\ 1.5,\ 1.75\ \text{days}$,
      respectively,
    shown from left to right and top to bottom.}\label{fig:ex2d_02_wcf}
  \end{figure}

  Figures~\ref{fig:ex2d_02_pressure_head} display the evolution of the
  pressure head over time within the heterogeneous two-layer medium. At
  early times ($t=0$), the pressure field is primarily determined by the
  initial hydrostatic distribution. As time progresses ($t=0.25$ to $1.75$
  days), the wetting front propagates downward from the upper boundary
  and interacts with the curved material interface. The variations in
  color illustrate the gradual redistribution of pressure across both
  soil layers, with the interface visibly influencing the shape of the
  evolving pressure contours.

  The snapshots in Figure~\ref{fig:ex2d_02_wcf} show the corresponding
  evolution of the volumetric water content. Initially, the upper and
  lower layers exhibit distinct water content distributions consistent
  with the prescribed initial condition. As infiltration proceeds, the
  region of higher water content expands and moves downward, following
  the development of the pressure field. The interface geometry induces
  observable changes in the spatial distribution of water content, and
  the wetting pattern progressively spreads through both soil layers
  over the simulated time interval.

  The numerical solutions exhibit smooth infiltration fronts across the
  heterogeneous interface and are in close agreement with the results
  reported in \cite{arxiv25}. Additionally, the total CPU
  time used in  this numerical example is $t = 2320\mathrm{s}$.

  \section{Conclusion}

  In this work we developed and analyzed a nonlinear solver for the
  Richards equation with strongly nonlinear constitutive relations. We
  constructed a conservative second-order in space and first order in
  time finite-difference
  algorithm and represent the resulting scheme in a compact
  algebraic form. For this nonlinear system,  we proposed a nonlinear
  Gauss-Seidel
  iteration with triangular splitting, diagonal shifts, and a
  stabilization parameter, leading to pointwise scalar updates at each
  grid node. Under mild monotonicity assumptions on the storage and
  sink terms, we established $L^\infty$-contraction of the scheme  and
  derived explicit conditions on the stabilization parameters in terms
  of the discrete diffusion operator. Finally, we combined the NGS
  iteration with a two-grid FAS procedure for the fully discrete system
  and illustrated its applicability to Richards-type test
  problems. This solvers could serve as the workhorse for the future
  inverse problems on estimating the conductivity or even the
  constitutive relation via the Richards equation.

  \section*{Acknowledgements}
  Xuelong Gu's research is supported by NSF award OIA-2242812.  Qi
  Wang’s research is partially supported by NSF awards
  DMS-2038080, OIA-2242812, and an SC GAIN-CRP award.

  \bibliographystyle{elsarticle-num}
  \bibliography{ref}

  \end{document}